\documentclass[11pt]{amsart}
\usepackage{amsfonts}
\usepackage[latin9]{inputenc}
\usepackage{amsmath}
\usepackage{amssymb}
\usepackage{breqn}
\usepackage{url}
\usepackage{graphicx}
\usepackage[pdfstartview=FitH]{hyperref}
\usepackage{amsthm}
\usepackage{authblk}
\usepackage{hyperref}
\usepackage{url}
\usepackage{cleveref}
\usepackage{verbatim}
\usepackage{bbm}

\calclayout

\crefformat{section}{\S#2#1#3} 
\crefformat{subsection}{\S#2#1#3}
\crefformat{subsubsection}{\S#2#1#3}

\newcommand{\im}{\operatorname{Im}}

\newcommand{\X}{\operatorname{X}}
\newcommand{\B}{\operatorname{\mathbb{E}}}

\newcommand{\id}{\operatorname{id}}

\newcommand{\sides}{\operatorname{sides}}

\newcommand{\Span}{\operatorname{span}}
\newcommand{\type}{\operatorname{type}}

\newcommand{\bipartite}{\operatorname{bipartite}}

\newcommand{\supp}{\operatorname{supp}}

\makeatletter
\begin{document}
\newtheorem{theorem}{Theorem}[section]
\newtheorem{lemma}[theorem]{Lemma}
\newtheorem{claim}[theorem]{Claim}
\newtheorem{proposition}[theorem]{Proposition}
\newtheorem{corollary}[theorem]{Corollary}
\theoremstyle{definition}
\newtheorem{definition}[theorem]{Definition}
\newtheorem{observation}[theorem]{Observation}
\newtheorem{example}[theorem]{Example}
\newtheorem{remark}[theorem]{Remark}

\title[Banach \.Zuk's criterion and random groups]{Banach \.Zuk's criterion for partite complexes with application to random groups}
\author{Izhar Oppenheim}
\newcommand{\Addresses}{{
  \bigskip
  \footnotesize

  IZHAR OPPENHEIM, \textsc{Department of Mathematics, Ben-Gurion University of the Negev, Be'er Sheva 84105, Israel}\par\nopagebreak
  \textit{E-mail address:} \texttt{izharo@bgu.ac.il}
}}

\maketitle
\begin{abstract}
We prove a Banach version of \.Zuk's criterion for groups acting on partite simplicial complexes. Using this new criterion we derive a new fixed point theorem for random groups in the Gromov density model with respect to several classes of Banach spaces ($L^p$ spaces, Hilbertian spaces, uniformly curved spaces).  In particular, we show that for every $p$, a group in the Gromov density model has asymptotically almost surely property $(F L^p)$ and give a sharp lower bound for the growth of the conformal dimension of the boundary of such group as a function of the parameters of the density model.
\end{abstract} 

\section{Introduction}

Fixed point properties of groups is a general theme is the study of groups.  For instance, for a discrete group, both amenability and Kazhdan's property (T) are equivalent to fixed point properties. This paper will focus on the study of fixed point properties of a group when acting on (a class of) Banach spaces. Explicitly, the following definition was given in Bader at el.  \cite{BFGM} as a generalization of property $(FH)$: Given a Banach space $\B$ and a group $G$, we say that $G$ has property $(F \B)$ if every continuous affine isometric action of $G$ on $\B$ has a fixed point. More generally, given a class of Banach spaces $\mathcal{E}$, we say that $G$ has property $(F \mathcal{E})$ if it has $(F \B)$ for every $\B \in \mathcal{E}$.  In this notation, property $(FH)$ is property $(F \mathcal{E}_{\text{Hilbert}})$, where $\mathcal{E}_{\text{Hilbert}}$ is the class of all Hilbert spaces.  Bader at el.  \cite{BFGM} also studied property $(F L^p)$ which will be considered in this article: Using the notation above,  property $(F L^p)$ is property $(F \mathcal{E}_{L^p})$ where $\mathcal{E}_{L^p}$ is the class of all $L^p$ spaces.

This paper was motivated by the questions regarding Banach fixed point properties for random groups.  Banach fixed properties for random groups was studied in \cite{DrutuM, LaatSalle, OppGarBan} - in all these works, the results regarding random groups in the triangular model were rather strong, e.g., all these works show that for every $2 \leq p < \infty$, a random group in the triangular model has property $(FL^p)$ asymptotically almost surely.  In contrast,  the results for random groups in the density model were much less satisfactory - \cite{LaatSalle, OppGarBan} did not address this model at all and in \cite{DrutuM} the results in this model were rather weak (see exact results below).  It was conjectured by Dru\c{t}u and Mackay \cite{DrutuM} that property $(FL^p)$ for every $2 \leq p < \infty$ should hold also in the density model.  The problem with generalizing the approach of the author in \cite{OppGarBan} to the density model was the lack of a  \.Zuk style criterion in the Banach setting for groups acting on partite complexes (see more on this below).  

The main results of this paper are: First, establishing a Banach \.Zuk criterion for groups acting on partite complexes (Theorem \ref{Zuk type thm intro}) and second, applying this criterion for random groups in the density model (Theorem \ref{f.p. density model - intro}). As a consequence, we also derive a sharp bound on the growth of the conformal dimension of the boundary of a random group in the density model (Theorem \ref{conformal dim growth thm}).

\subsection{Banach \.Zuk's criterion for groups acting on partite simplicial complexes}

In  \cite{Gar}, Garland showed how to prove vanishing of real cohomology for a group acting on an affine building via studying the local geometry of the building.  This approach was later generalized by various authors to vanishing of cohomology with Hilbert coefficients for groups acting on simplicial complexes - see \cite{Zuk, BS,DJ1, DJ2, OppWeighted}. We recall that vanishing of the first cohomology with Hilbert coefficients is equivalent to property (FH). Later, this approach was also applied in the Banach setting in various levels of effectiveness/success - see  \cite{Bourd, OppFP, Nowak, OppRobust,  OppVanBan,DrutuM, LaatSalle, OppGarBan}.  

The idea behind this approach is the following local to global argument: Given a group $G$ acting geometrically on a simplicial complex $X$, one can deduce fixed point properties of $G$ by geometric conditions on the links on $X$.  In dimension $2$, this approach is also known as \.Zuk's criterion.   Recently, there were also computer assisted proofs for property (T) and property (FH) that refined this idea and gave geometric proofs studying the geometry of $2$-balls around vertices (while \.Zuk's criterion studies the geometry of $1$ balls) - see \cite{AutFn1, AutFn2, AutFn3}.

In order to explicitly state \.Zuk's criterion and our generalization of it, we will need some terminology of spectral graph theory.  Given a finite graph $(V,E)$, denote $m(v)$ to be the valency of $v \in V$. We also define $\ell^2 (V,m)$ as to be the space of functions $\phi: V \rightarrow \mathbb{C}$ with an inner-product
$$\langle \phi, \psi \rangle = \sum_{\lbrace v \rbrace \in V} = m(v) \phi (v) \overline{\psi (v)}.$$

The \textit{simple random walk} on $(V,E)$ as above is the operator $A: \ell^2 (V,m) \rightarrow \ell^2 (V,m)$ defined as  
$$(A \phi) (v) = \sum_{u \in V, \lbrace u, v \rbrace \in E} \frac{m (\lbrace u, v \rbrace)}{m(v)} \phi (u).$$
With the inner-product above, this is a self-adjoint operator and its spectrum is contained in $[-1,1]$.  We recall that the space of constant functions is an eigenspace of $A$ with eigenvalue $1$ and, if $(V,E)$ is connected,  then all other the other eigenvalues of $A$ are strictly less than $1$.

The Hilbert (classical) version of \.Zuk's criterion goes as follows:
\begin{theorem}[\cite{BS, Zuk} for $n =2$, \cite{OppWeighted} for all $n \geq 2$]
Let $G$ be a locally compact, unimodular group and $X$ be a pure $n$-dimensional simplicial complex such that $X$ is gallery connected and all the $1$-dimensional links of $X$ are connected.  Assume that $G$ is acting on $X$ by simplicial automorphisms and the the action is proper and cocompact.  If for every $1$-dimensional link of $X$, the second largest eigenvalue of the simple random walk on the link is $< \frac{1}{n+1}$, then $G$ has property $(FH)$ (and property (T)).
\end{theorem}

We note that in the Hilbert case, the criteria stated above need only a bound on the second largest eigenvalue of the links. However, when generalizing this criterion to the Banach setting, one usually needs to bound the norm of the random walk operator in order to get a workable version - see \cite{LaatSalle, OppGarBan}.  Explicitly,  for a finite graph $(V,E)$ and a Banach space $\B$, we define $\ell^2 (V, m ; \B)$ to be the space of functions $\phi: V \rightarrow \B$ with the norm 
$$\Vert \phi \Vert^2 = \sum_{v \in V} m(v) \vert \phi (v) \vert_{\B}^2.$$
Define $A \otimes \id_{\B} : \ell^2 (V, m ; \B) \rightarrow \ell^2 (V, m ; \B)$ as above and $M \otimes \id_{\B} : \ell^2 (V, m ; \B) \rightarrow \ell^2 (V, m ; \B)$ as 
$$M \phi \equiv \frac{1}{\sum_{v \in V} m(v)} \phi (v).$$
Denote $\lambda_{(V,E)}^{\B} = \Vert (A (I-M)) \otimes \id_{\B} \Vert_{B (\ell^2 (V, m ; \B))}$.  Note that in the Hilbert setting (where $\B = \mathcal{H}$), bounding $\lambda_{(V,E)}^{\mathcal{H}}$ is equivalent to bounding the non-trivial spectrum of the operator from above and below.   In \cite{OppGarBan}, the author showed the following:
\begin{theorem}\cite[Theorem 4.4]{OppGarBan}
Let $G$ be a unimodular, locally compact group that acts properly and cocompactly on a locally finite, pure $2$-dimensional simply connected simplicial complex $X$ with connected $1$-dimensional links.   For any reflexive Banach space $\B$, if 
$$\max_{v \in X(0)} \lambda_{X_v}^{\B} < \frac{1}{2}$$
then $G$ has property $(F \B)$.
\end{theorem}
 
The main issue with using the operator norm of the random walk instead to the second largest eigenvalue is that in some interesting examples the simplicial complex $X$ is partite.  An $n$-dimensional simplicial complex is called \textit{partite} (or colorable) if its vertices can be colored in $n+1$ different colors such that each $n$-simplex has a vertex of every color.  In particular, a graph $(V,E)$ is called \textit{bipartite} if $V$ can be partitioned into two disjoint sets $V = S_1 \cup S_2$ called \textit{sides} such that for each $\lbrace u,v \rbrace \in E$, $\vert \lbrace u,v \rbrace \cap S_1 \vert =  \vert \lbrace u,v \rbrace \cap S_2 \vert =1$, i.e., each edge has exactly one vertex in each side.  We note that for a partite $n$-dimensional complex,  each one dimensional link is a bipartite graph.  We also note that the random walk on bipartite graphs always has the eigenvalue $-1$ corresponding to the function $ \mathbbm{1}_{S_1} - \mathbbm{1}_{S_2}$ (where $\mathbbm{1}_{S_i}$ is the indicator function on $S_i$) and thus for every (non zero dimensional) Banach space $\B$ we always have $\lambda_{(V,E)}^\B \geq 1$.  

Consider the classical case of $\B = \mathbb{C}$. For a bipartite graph $(V,E)$,  when $A$ acts on $\ell^2 (V, m)$,  it is standard to think of the the trivial spectrum of $A$ as $\pm 1$ and the space of trivial eigenfunctions is the space spanned by the constant function and $\mathbbm{1}_{S_1} - \mathbbm{1}_{S_2}$ or equivalently the space of trivial eigenfunctions is $\Span \lbrace \mathbbm{1}_{S_1}, \mathbbm{1}_{S_2} \rbrace$.  The projection on the space of trivial eigenfunctions is can be described explicitly as follows: We define the following averaging operators
$M_{1}, M_{2} : \ell^2 (V, m) \rightarrow \mathbb{C}$: 
$$M_{i} \phi = \frac{1}{m (S_{i})} \sum_{u \in S_{i}} \phi (u),$$
where $m (S_i) = \sum_{v \in S_i} m(v)$.  We also define $M_{\sides} : \ell^2 (V, m) \rightarrow \ell^2 (V, m)$ by 
$$M_{\sides} \phi (u) = 
\begin{cases}
M_{1} \phi & u \in S_{1} \\
M_{2} \phi & u \in S_{2} 
\end{cases}.$$ 
With this notation, the non-trivial spectrum of $A$ is bounded in the interval $[- \Vert A(I-M_{\sides}) \Vert, \Vert A(I-M_{\sides}) \Vert]$.  This discussion gives rise to the following definition: Let $\B$ be a Banach space and $(V,E)$ be a bipartite graph. Define 
$$\lambda_{(V,E),\bipartite}^\B = \Vert (A (I-M_\sides)) \otimes \id_{\B} \Vert_{B(\ell^2 (V,m ; \B))}.$$
For a partite $n$-dimensional $X$ and an $(n-2)$-simplex $\tau$, we denote $\lambda_{\tau,   \bipartite}^\B = \lambda_{X_\tau,   \bipartite}^\B$ where $X_\tau$ is the link of $\tau$ (which is a bipartite graph).  We these definitions, we prove the following Theorem: 
\begin{theorem}[\.Zuk type criterion for Banach fixed point property - see Theorem \ref{Zuk type thm}  for a more detailed version]
\label{Zuk type thm intro}
Let $G$ be a locally compact, unimodular group and $X$ be a pure $n$-dimensional, partite simplicial complex such that $X$ is gallery connected and all the $1$-dimensional links of $X$ are connected.  Assume that $G$ is acting on $X$ by simplicial automorphisms and the the action is proper and cocompact.   Let $\mathcal{E}_X$ be the class of Banach spaces such that for every $\B \in \mathcal{E}_X$ it holds that
$$\max_{\tau \in X (n-2)} \lambda_{\tau,   \bipartite}^\B  < \frac{1}{8n-3}.$$
Assume that $\mathbb{C} \in  \mathcal{E}_X$, then $G$ property $(F \mathcal{E}_X)$. 
\end{theorem}

The idea behind the proof of this Theorem is the following: In \cite{OppRobust}, the author proved a version of this Theorem under the assumption that the fundamental domain of the action is a single simplex. The main argument there was bounding the angle between the projections defined by the subgroups stabilizing the $(n-1)$-faces of the fundamental domain. In this paper, the idea is to use this angle criterion, but apply it on the projections defined by the coloring of the simplicial complex. As in \cite{OppRobust}, this leads to a proof of a strengthened version of Banach property (T) (i.e., robust property (T)) and this property in turn leads to Banach fixed point properties.    

\subsection{Fixed point properties of a random group in the Gromov model}

A random group is a group chosen randomly according to some model and one is interested in the asymptotic properties of such randomly chosen group.  The most famous model is the Gromov density model: 
\begin{definition}[Gromov density model]
Let $k \in \mathbb{N}, k \geq 2$ and $0 \leq d \leq 1$ be constants and $l \in \mathbb{N}$ be a parameter.  A random group is the Gromov density model $\mathcal{D} (k,l,d)$ is a group $\Gamma = \langle \mathcal{A}  \vert \mathcal{R} \rangle$ where $\vert \mathcal{A}  \vert =k$ and $\mathcal{R}$ is a set of relators of length $l$ (in $\mathcal{A} \cup \mathcal{A} ^{-1}$) randomly chosen from the set 
$$\lbrace \mathcal{R} \text{ is a set of cyclically reduced relators of length } l :  \vert \mathcal{R} \vert = \lfloor (2k-1)^{dl} \rfloor \rbrace$$ with uniform probability. We denote a random group in this model by $\Gamma \in \mathcal{D} (k,l,d)$.

For a group property $P$, we say that $P$ holds asymptotically almost surely (a.a.s.) in $\mathcal{D} (k,l,d)$ if 
$$\lim_{l \rightarrow \infty} \mathbb{P} (\Gamma \in \mathcal{D} (k,l,d) \text{ has property } P) = 1.$$
\end{definition}

It was proven \cite{Zuk2, KK, DrutuM, Ash2}, that for every $d > \frac{1}{3}$ property (T) (and equivalently, property (FH)) holds  a.a.s.  in $\mathcal{D} (k,l,d)$  (in most cases, it is assumed that $l$ is divisible by $3$, but recently Ashcroft \cite{Ash2} showed how to remove this assumption).  Prior to our work, generalizing this result to $L^p$ spaces was met with some difficulty: Indeed,  prior to this work, the state-of-the-art result was by Dru\c{t}u and Mackay \cite{DrutuM} who proved the following: Given $2 \leq p < \infty$, if $k \geq 10 \cdot 2^p$, then for every $d > \frac{1}{3}$ it holds a.a.s. that $\mathcal{D} (k,l,d)$ has property $(FL^{p'})$ for any $2 \leq p' \leq p$. As noted by Dru\c{t}u and Mackay \cite{DrutuM} it is natural to expect that for any $k \geq 2$ and any $d > \frac{1}{3}$ it holds a.a.s. that $\mathcal{D} (k,l,d)$ has property $(FL^{p})$ for any $2 \leq p < \infty$, but they could not achieve this result. 

Below, we use our Theorem \ref{Zuk type thm intro} to vastly improve on the known results for Banach fixed point properties for random group in the Gromov density model.  Explicitly, we consider the class of uniformly curved Banach spaces (see Definition \ref{uc def} below) and prove that for every $k \geq 2$ and every $d > \frac{1}{3}$ it holds a.a.s. that $\mathcal{D} (k,l,d)$ has property $(F\B)$ for every uniformly curved Banach space. The class of uniformly curved Banach spaces contains all $L^p$ spaces and thus we prove the fixed point $L^p$ property that was conjectured by Dru\c{t}u and Mackay. More explicitly, we prove the following:

\begin{theorem}[Banach fixed point properties for the density model - see Theorem \ref{f.p. density model}  for a more detailed version]
\label{f.p. density model - intro}
Let $\frac{1}{3} <d < \frac{1}{2}$, $k \geq 2$ be constants.  For $l$ divisible by $3$,  it holds a.a.s. that  $\mathcal{D} (k,l,d)$ has property $(F\B)$ for any uniformly curved space $\B$.  In particular, for every $2 \leq p < \infty$  it holds a.a.s. that $\mathcal{D} (k,l,d)$ has property $(F L^p)$.  Explicitly,  there are universal constants $C' , C''$ independent of $k,l,d$ such that for any
$$2 \leq p \leq  C ' (d-\frac{1}{3}) (\log (2k-1)) l - C'',$$
$\mathcal{D} (k,l,d)$ has property $(F L^p)$ a.a.s.  (given that $l$ is divisible by $3$).
\end{theorem}

As a corollary, we give a lower bound on the conformal dimension of the boundary of a random groups in the density model.  Namely, by a Theorem of Bourdon  \cite{Bourd2}, if for a given $2 \leq p$, a hyperbolic group $\Gamma$ has property $(F L^p)$, then the conformal dimension of $\partial_\infty \Gamma$ is $\geq p$. This readily gives the following:
\begin{corollary}
\label{conformal dim growth corollary}
Let $\frac{1}{3} <d < \frac{1}{2}$, $k \geq 2$ be constants.  For $l$ divisible by $3$ let $\Gamma$ be a random group in the model $\mathcal{D} (k,l,d)$.  Then there are universal constants $C', C''$ independent of $k,l,d$ such that it holds a.a.s. that 
$$C ' (d-\frac{1}{3}) (\log (2k-1)) l - C'' \leq {\rm Confdim} (\partial_\infty \Gamma) .$$
\end{corollary}

An upper bound for the conformal dimension was given by Mackay:
\begin{proposition} \cite[Proposition 1.7]{Mackay}
Let $\frac{1}{3} <d < \frac{1}{2}$, $k \geq 2$ be constants and let $\Gamma$ be a random group in the model $\mathcal{D} (k,l,d)$.  Then it holds a.a.s. that 
$${\rm Confdim} (\partial_\infty \Gamma) \leq \frac{16}{\log (2)} \frac{1}{1-2d} (\log (2k-1)) l..$$
\end{proposition}
Combining Corollary \ref{conformal dim growth corollary} with Mackay's upper bound,  one can see that our lower bound for the conformal dimension is in fact sharp:

\begin{theorem}
\label{conformal dim growth thm}
Let $\frac{1}{3} <d < \frac{1}{2}$, $k \geq 2$ be constants.  For $l$ divisible by $3$, let $\Gamma$ be a random group in the model $\mathcal{D} (k,l,d)$.  Then there are universal constants $C', C''$ independent of $k,l,d$ such that it holds a.a.s. that 
$$C ' (d-\frac{1}{3})  - \frac{C''}{(\log (2k-1)) l} \leq \frac{{\rm Confdim} (\partial_\infty \Gamma)}{(\log (2k-1)) l}  \leq \frac{16}{\log (2)} \frac{1}{1-2d} .$$
\end{theorem}

\begin{remark}
After the completion of this manuscript, we were informed about a forthcoming work by Jordan Frost attaining a similar lower bound on the growth of the conformal dimension for all $d < \frac{1}{2}$ (and not just $\frac{1}{3} <d < \frac{1}{2}$ as in our work). We note that Frost methods are completely different from ours and he does not attain results regarding Banach fixed point properties. 
\end{remark}

Our proof of Theorem \ref{f.p. density model - intro} follows the scheme of proof \cite[Theorem B]{KK}: first, we reduce the problem to another model of random groups in which the Cayley complex of a random group is a partite simplicial complex and then we apply our version of \.Zuk's criterion (this is not completely straightforward and we heavily use the ideas of \cite{ALS} in the analysis). 

\subsection{Organization} 
The paper is organized as follows. In Section \ref{Prelim. sec}, we cover preliminary material. In Section \ref{Zuk criterion sec}, we prove  Banach \.Zuk's criterion for groups acting on partite complexes,  i.e., we prove an extended version of Theorem \ref{Zuk type thm intro}. In Section \ref{random groups sec}, we use our criterion to prove Banach fixed point properties for random groups in the density model, i.e., we prove an extended version of Theorem \ref{f.p. density model - intro}.

\section{Preliminaries}
\label{Prelim. sec}

\subsection{Robust Banach property (T) and the fixed point property}

In \cite{OppRobust}, the author introduced the notion of robust Banach property (T) as a strengthened version of Banach property (T) (as defined is \cite{BFGM}) that is weaker than the notion of strong Banach property (T) defined by V. Lafforgue \cite{Laff}  

Let $G$ be a locally compact group. Let $\mathcal{F}$ be a family of linear representations on Banach spaces, $\pi : G \rightarrow B (\B)$, that are continuous with respect to the strong operator topology.   Denote $C_c (G)$ to be the compactly supported continuous functions $f: G \rightarrow \mathbb{C}$.  For every $\pi \in \mathcal{F}$, define $\pi (f) \in B(\B)$ via the Bochner integral 
$$\pi (f).x = \int_{G} f(g) \pi (g).x dg$$ 
 (recall that $f$ is compactly supported and thus this integral converges).  Define the norm $\Vert . \Vert_{\mathcal{F}}$ on $C_c (G)$ as $\Vert f \Vert_{\mathcal{F}} =  \sup_{\pi \in \mathcal{F}} \Vert \pi (f) \Vert$.

If $\mathcal{F}$ is closed under complex conjugation (i.e., $\pi \in \mathcal{F} \Rightarrow \overline{\pi} \in \mathcal{F}$) and under duality (i.e., $\pi \in \mathcal{F} \Rightarrow \pi^* \in \mathcal{F}$), then $C_{\mathcal{F}} (G)$ is a Banach algebra with an involution 
$$f^* (g) = \overline{f(g^{-1})}, \forall g \in G.$$

\begin{definition}[Robust Banach property (T)]
\label{robust property T definition - compact generation}
Let $G$ be a compactly generated group and let $K$ be some symmetric compact set that generates $G$. For a class of Banach spaces $\mathcal{E}$ and a constant $\beta \geq 1$, denote $\mathcal{F} (\mathcal{E}, K,\beta)$ to be the class of all the continuous representations $\pi$ of $G$ on some $\B \in \mathcal{E}$ such that $\sup_{g \in K} \Vert \pi (g) \Vert \leq \beta$. 

We say that $G$ has robust Banach property (T) with respect to a class of Banach spaces $\mathcal{E}$, if there exists $\beta >1$ and a sequence of real functions $f_k \in C_c (G)$ such that for every $k$, $\int f_k=1$ and such that the sequence $(f_k)$ converges in $C_{\mathcal{F} (\mathcal{E},K,\beta)}$ to $p$ and $\forall (\pi, \B) \in \mathcal{F} (\mathcal{E},K,\beta)$, $\pi (p)$ is a projection on $\B^{\pi (G)}$.

We will call the sequence $f_k$ above a Kazhdan projection with respect to $\mathcal{F} (\mathcal{E},K,\beta)$.
\end{definition}

\begin{remark}
The above Definition assumes compact generation. A more general definition can be found in \cite{OppRobust}. Also, the definition in \cite{OppRobust} assumes that the functions $f_k$ are symmetric and thus (as noted in \cite{SalleLocal}) $p$ is central. Since the main focus of this paper is fixed point properties (see below), centrality of $p$ is not needed and thus omitted from the definition.
\end{remark}

Robust Banach property (T) is connected to the fixed point property defined as follows:
\begin{definition}
For a Banach space $\B$,  we say that $G$ has property $(F \B)$ if every affine isometric action of $G$ on $\B$ has a fixed point.  For a class of Banach spaces $\mathcal{E}$, we say that $G$ has property $(F \mathcal{E})$ if for every $\B \in \mathcal{E}$,  $G$ has property $(F \B)$. In particular for $1 \leq p \leq \infty$, a group $G$ is said property $(F L^p)$ if it has property $(F \B)$ for every $\B$ that is an $L^p$-space.
\end{definition}

\begin{proposition}\cite[Proposition 5.9]{OppRobust}
\label{robust implies fp prop}
Let $\mathcal{E}$ be a class of Banach spaces such that for every $\B \in \mathcal{E}$ it holds that $\B \oplus_{\ell^2} \mathbb{C} \in \mathcal{E}$. If $G$ has robust Banach property (T) with respect $\mathcal{E}$, then $G$ has property $(F \mathcal{E})$. 
\end{proposition} 

It is worth noting that under some extra assumptions, then opposite direction of this Proposition is also true (see \cite[Corollary 5.6]{SalleLocal}), but we will make no use of this fact.  

Last, we state some hereditary properties of property $(F \mathcal{E})$. 
\begin{proposition}
\label{f.p preserve under quotient prop}
If $G$ has property $(F \mathcal{E})$ and $G'$ is a quotient of $G$, then $G'$ has property $(F \mathcal{E})$.
\end{proposition}

\begin{proof}
Every isometric action of $G'$ on a Banach space $\B$ induces an isometric action of $G$ on $\B$.
\end{proof}

Also,  under suitable assumptions, a group $G$ has property $(F \mathcal{E})$ if and only if a finite index subgroup $H < G$ has property $(F \mathcal{E})$:
\begin{proposition}\cite[Section 8]{BFGM}
\label{f.p. is equivalent to f.i. prop}
Let $G$ be as above and $H <G$ be a closed finite index subgroup of $G$.  For every super-reflexive space, if $H$ has property $(F \B)$, then $G$ has property $(F \B)$. Conversely,  let $\mathcal{E}$ be a class of Banach spaces such that every $\B \in \mathcal{E}$ is super-reflexive and such that $\mathcal{E}$ is closed under $\ell^2$ sums. Then if $G$ has property $(F \mathcal{E})$ it follows that $H$ has property $(F \mathcal{E})$.
\end{proposition}

\subsection{Vector valued $\ell^2$ spaces}
\label{Vector valued subsec}

Given a finite set $V$, a function  $m: V \rightarrow \mathbb{R}_+$ and a Banach space $\B$, we define the \textit{vector valued space} $\ell^2 (V, m ; \B)$ to be the space of functions $\phi : V \rightarrow \B$, with the norm 
$$\Vert \phi \Vert_{\ell^2 (V,m ; \B)} = \left( \sum_{v \in V} m(v) \vert \phi (v) \vert^2 \right)^{\frac{1}{2}},$$
where $\vert . \vert$ is the norm of $\B$. We denote $\ell^2 (V,m) = \ell^2 (V,m ; \mathbb{C})$ and recall that $\ell^2 (V,m)$ is also a Hilbert space with the inner-product
$$\langle \phi, \psi \rangle = \sum_{v  \in V} m(v) \phi (v) \overline{\psi (v)}.$$

Let $T: \ell^2 (V,m) \rightarrow \ell^2 (V,m)$ be a linear operator and $(T_{v,u})_{u,v \in V} \in M_{\vert V \vert} (\mathbb{C})$ be the matrix such that for every $\phi \in \ell^2 (V,m)$ it holds that 
$$(T \phi) (v) = \sum_{u \in V} T_{v,u} \phi (u).$$
Define $T \otimes \id_{\B} : \ell^2 (V,m ; \B) \rightarrow \ell^2 (V,m ; \B)$ by the formula:
$$((T \otimes \id_{\B}) \phi) (v) = \sum_{u \in V} T_{v,u} \phi (u),$$
for every $\phi \in \ell^2 (V,m ; \B)$. We denote $\Vert T \otimes \id_{\B} \Vert_{B( \ell^2 (V,m ; \B))}$ to be the operator norm of $T \otimes \id_{\B}$.

The norm $\Vert T \otimes \id_\B \Vert_{B( \ell^2 (V,m ; \B))}$ is preserved under some operations on $\B$ - this is summed up in the following Lemma:
\begin{lemma}
\label{L2 norm stability}
Let $V$ be a finite set, $T$ a bounded operator on $\ell^2 (V, m)$ and $C >0$ constant. Let $\mathcal{E} = \mathcal{E} (C)$ be the class of Banach spaces defined as:
$$\mathcal{E} = \lbrace \B : \Vert T \otimes \id_\B \Vert_{B(\ell^2 (V, m ; \B))} \leq C \rbrace.$$
Then this class is closed under quotients, subspaces, $\ell^2$-sums and ultraproducts of Banach spaces, i.e., preforming any of these operations on Banach spaces in $\mathcal{E}$ yield a Banach space in $\mathcal{E}$. 
\end{lemma}

\begin{proof}
The fact that $\mathcal{E}$ is closed under quotients, subspaces and ultraproducts of Banach spaces was shown in \cite[Lemma 3.1]{Salle}. The fact that $\mathcal{E}$ is closed under $\ell^2$-sums is straight-forward and left for the reader.
\end{proof}

\subsection{Uniformly curved Banach spaces}

Uniformly curved Banach spaces were introduced by Pisier in \cite{Pisier2}:

\begin{definition}[Fully contractive operator]
An operator $T: \ell^2 (V,m) \rightarrow \ell^2 (V,m)$ is called fully contractive if for every Banach space $\B$ it holds that $\Vert T \otimes \id_{\B} \Vert_{B( \ell^2 (V,m ; \B))} \leq 1$. 
\end{definition}

\begin{definition}[Uniformly curved space]
\label{uc def}
Let $\B$ be a Banach space. The space $\B$ is called uniformly curved if for every $0 < \varepsilon \leq 1$ there is $\delta >0$ such that for every space $\ell^2 (V,m)$ and every fully contractive linear operator $T:\ell^2 (V,m) \rightarrow \ell^2 (V,m)$, if $\Vert T \Vert_{B(\ell^2 (V,m))} \leq \delta$, then $\Vert T \otimes \id_{\B} \Vert_{B(\ell^2 (V,m ; \B))} \leq \varepsilon$. 
\end{definition}


Given a monotone increasing function $\omega: (0,1] \rightarrow  (0,1]$ such that 
$$\lim_{t \rightarrow 0^{+}} \omega (t) = 0,$$
we denote 
$\mathcal{E}^{\text{u-curved}}_\omega$ to be the class of all uniformly curved Banach spaces $\B$ such that for every space $\ell^2 (V,m)$ and every fully contractive linear operator $T:\ell^2 (V,m) \rightarrow \ell^2 (V,m)$, if $\Vert T \Vert_{B(\ell^2 (V,m))} \leq \delta$, then $\Vert T \otimes \id_{\B} \Vert_{B(\ell^2 (V,m ; \B))} \leq \omega (\delta)$. 

Applying Lemma \ref{L2 norm stability} on $\mathcal{E}^{\text{u-curved}}_\omega$ defined above yields the following Corollary:
\begin{corollary}
For any monotone increasing function $\omega: (0,1] \rightarrow  (0,1]$ such that $\lim_{t \rightarrow 0^{+}} \omega (t) = 0$, the class $\mathcal{E}^{\text{u-curved}}_\omega$ defined above is closed under quotients, subspaces, $\ell^2$-sums and ultraproducts of Banach spaces.
\end{corollary}

\begin{proposition}
\label{L contractive prop}
Let $T :\ell^2 (V,m) \rightarrow \ell^2 (V,m)$ be a linear operator and $L\geq 1$, $0< \delta \leq 1$ be constants such that:
\begin{enumerate}
\item It holds that $\Vert T \Vert_{B(\ell^2 (V,m))} \leq \delta$.
\item For every Banach space $\B$, $\Vert T \otimes \id_{\B} \Vert_{B(\ell^2 (V,m ; \B))} \leq L$.
\end{enumerate}
Then for every monotone increasing function $\omega: (0,1] \rightarrow  (0,1]$ such that $\lim_{t \rightarrow 0^{+}} \omega (t) = 0$ and every $\B \in \mathcal{E}^{\text{u-curved}}_\omega$, $\Vert T \otimes \id_{\B} \Vert_{B(\ell^2 (V,m ; \B))} \leq L \omega (\delta)$.
\end{proposition}

\begin{proof}
We note that $\frac{1}{L} T$ is a fully contractive operator such that 
$$\Vert \frac{1}{L} T \Vert_{B(\ell^2 (V,m))} \leq \frac{\delta}{L}.$$ 
Thus, by the definition of $\mathcal{E}^{\text{u-curved}}_\omega$ it follows for every $\B \in \mathcal{E}^{\text{u-curved}}_\omega$ that 
$$\Vert (\frac{1}{L}) T \otimes \id_{\B} \Vert_{B(\ell^2 (V,m ; \B))} \leq \omega (\frac{\delta}{L}),$$
and thus 
$$\Vert T \otimes \id_{\B} \Vert_{B(\ell^2 (V,m ; \B))} \leq L \omega (\frac{\delta}{L}) \leq L \omega (\delta),$$
where the last inequality is due to the fact that $L \geq 1$ and $\omega$ is monotone increasing.
\end{proof}

\subsection{Strictly $\theta$-Hilbertian spaces}
\label{subsect Strictly theta-Hilbertian spaces}

Here we will describe a special class of uniformly curved Banach spaces that contains all (commutative and non-commutative) $L^p$ spaces. 

Two Banach spaces $\B_0, \B_1$ form a \textit{compatible pair} $(\B_0,\B_1)$ if they are continuously linear embedded in the same topological vector space. The idea of complex interpolation is that given a compatible pair $(\B_0,\B_1)$ and a constant $0 \leq \theta \leq 1$, there is a method to produce a new Banach space $[\B_0, \B_1]_\theta$ as a ``convex combination'' of $\B_0$ and $\B_1$. We will not review this method here, and the interested reader can find more information on interpolation in \cite{InterpolationSpaces}.

This brings us to consider the following definition due to Pisier \cite{Pisier}: a Banach space $\B$ is called \textit{strictly $\theta$-Hilbertian} for $0 < \theta \leq 1$, if there is a compatible pair $(\B_0,\B_1)$ with $\B_1$ a Hilbert space such that $\B = [\B_0, \B_1]_\theta$. Examples of strictly $\theta$-Hilbertian spaces are $L^p$ space and non-commutative $L^p$ spaces (see \cite{PX} for definitions and properties of non-commutative $L^p$ spaces), where in these cases $\theta = \frac{2}{p}$ if $2 \leq  p  < \infty $ and $\theta = 2 - \frac{2}{p}$ if $1 < p \leq 2$. 

For our use, it will be important to bound the norm of an operator of the form $T \otimes \id_{\B}$ given that $\B$ is an interpolation space.

\begin{lemma} \cite[Lemma 3.1]{Salle}
\label{interpolation fact}
Let $(\B_0,\B_1)$ be a compatible pair , $V$ be a finite set, $m: V \rightarrow \mathbb{R}_+$ be a function and $T \in B(\ell^2 (V,m))$ be an operator. Then for every $0 \leq \theta \leq 1$,
$$\Vert T \otimes \id_{[\B_0, \B_1]_\theta} \Vert_{B(\ell^2 (V,m; [\B_0, \B_1]_\theta))} \leq \Vert T \otimes \id_{\B_0} \Vert_{B(\ell^2 (V,m ; \B_0))}^{1-\theta} \Vert T \otimes \id_{\B_1} \Vert_{B(\ell^2 (V,m; \B_1))}^{\theta},$$
where $[\B_0, \B_1]_\theta$ is the interpolation of $\B_0$ and $\B_1$.
\end{lemma}

This Lemma has the following Corollary that shows that strictly $\theta$-Hilbertian spaces are uniformly curved (see also \cite[Lemma 3.1]{Salle}):

\begin{corollary}
\label{norm bound on theta-Hil coro}
Let $\B$ be a strictly $\theta$-Hilbertian space with $0 < \theta \leq 1$, $V$ be a finite set, $m: V \rightarrow \mathbb{R}_+$ be a function and $0 < \delta < 1$ be a constant. Assume that $T \in B(\ell^2 (V,m))$ is a fully contractive operator such that $\Vert T \Vert_{B(\ell^2 (V, m))} \leq \delta$. Then $\Vert T \otimes \id_{\B} \Vert_{B(\ell^2 (V, m ; \B))} \leq \delta^{\theta}$.

In other words, if $\B$ is strictly $\theta$-Hilbertian space with $0 < \theta \leq 1$, then for $\omega (t) = t^{\theta}$, we have that $\B \in \mathcal{E}^{\text{u-curved}}_\omega$.
\end{corollary}

\begin{proof}
For every Hilbert space $\B_1$ we have that $\Vert T \otimes \id_{\B_1} \Vert_{B(\ell^2 (V, m ; \B_1))} \leq \delta$ and thus the assertion stated above follows from Lemma \ref{interpolation fact}.
\end{proof}

Combining Lemma \ref{L2 norm stability} with the above Corollary yields: 

\begin{corollary}
\label{theta-hil is in uni-curv cor}
For a constant $0 < \theta_0 \leq 1$, denote $\mathcal{E}_{\theta_0}$ to be the smallest class of Banach spaces that contains all strictly $\theta$-Hilbertian Banach spaces for all $\theta_0 \leq \theta \leq 1$ and is closed under subspaces, quotients, $\ell^2$-sums and ultraproducts of Banach spaces. Then for every  $0 < \theta_0 \leq 1$, we have that $\mathcal{E}_{\theta_0} \subseteq \mathcal{E}^{\text{u-curved}}_{\omega (t) = t^{\theta_0}}$. 
\end{corollary}

\begin{remark}
A deep result of Pisier shows that the converse of the above Corollary is ``almost true'' if one considers arcwise $\theta_0$-Hilbertian spaces (see definition in \cite[Section 6]{Pisier2}). Namely, by \cite[Corollary 6.7]{Pisier2}, for every $\theta_0 < \theta \leq 1$ it holds that every Banach space in $\mathcal{E}^{\text{u-curved}}_{\omega (t) = t^{\theta}}$ is a subquotient of an arcwise $\theta_0$-Hilbertian space. We will not define arcwise $\theta_0$-Hilbertian spaces here and we will make no use of this fact.
\end{remark}

\subsection{Angle between projections}

The notion of an angle between projection was defined by the author in \cite{OppRobust} and further developed in \cite{OppVanBan} and\cite{OppAngle}. Below, we give the definitions and results on this subject that are needed for this paper. 

\begin{definition}
\label{Angle between proj def}
Let $\B$ be a Banach space and let $P_1, P_2$ be projections in $B (\B)$.  Assume that there is a projection $P_{1,2}$ on $\im (P_{1}) \cap \im (P_{2})$, such that $P_{1,2} P_{1} = P_{1,2}, P_{1,2} P_{2} = P_{1,2}$. Define the cosine of the angle between $P_1, P_2$ (with respect to $P_{1,2}$) as 
$$\cos_{P_{1,2}} (\angle (P_1,P_2)) = \max \lbrace \Vert P_1 P_2 - P_{1,2} \Vert,  \Vert P_2 P_1 - P_{1,2} \Vert \rbrace.$$
\end{definition}

We note that the angle between $P_1, P_2$ depends on the choice of $P_{1,2}$ and different choices yield different angles (see \cite[Example 2]{BKRS}).  However, to avoid cumbersome notation,  when $P_{1,2}$ is obvious from the context, we will denote $\cos (\angle (P_1,P_2))$ instead of $\cos_{P_{1,2}} (\angle (P_1,P_2))$.

\begin{theorem}\cite[Theorem 3.12]{OppRobust}
\label{angle criterion thm}
Let $\B$ be a Banach space and let $P_0,...,P_n$ be projections in $B(\B)$ ($n \geq 1$). Assume that for every $0 \leq j_1 < j_2 \leq n$, there is a projection $P_{j_1,j_2}$ on $\im (P_{j_1}) \cap \im (P_{j_2})$, such that $P_{j_1,j_2} P_{j_1} = P_{j_1,j_2}, P_{j_1,j_2} P_{j_2} = P_{j_1,j_2}$. Denote $T = \sum_k \frac{1}{n+1} P_k$. Assume further that there are constants 
$$\gamma < \frac{1}{8n-3} \text{ and } \beta < 1+ \frac{1-(8n-3)\gamma}{n-1 + (3n-1)\gamma},$$
such that 
$$\max_{0 \leq j \leq n} \Vert P_j \Vert \leq \beta$$
and
$$\max_{0 \leq j_1 < j_2 < \leq n} \cos  (\angle (P_{j_1},P_{j_2})) \leq \gamma.$$
Then there is a projection $T^\infty$ on $\bigcap_{j=0}^n \im (P_j)$ and constants $0\leq r (\gamma, \beta )  <1, C (\gamma, \beta )  >0$, such that for every $i$, 
$$\Vert T^i - T^\infty \Vert \leq Cr^i$$
and in particular, $T^i$ converges to $T^{\infty}$ in the operator norm. 
\end{theorem}

\begin{remark}
The conditions for the convergence of $T^i$ stated above are not optimal and can be somewhat improved using \cite{OppAngle}[Theorem 2.2]. 
\end{remark}






\subsection{Random walks on bipartite finite graphs}
\label{Random walks on bipartite finite graphs subsec}

Here we collect some facts regrading $\lambda_{(V,E), \bipartite}^{\B}$ defined in the introduction.  To ease the reading, we will repeat the definition and also repeat several definitions and facts that were already mentioned in the introduction. 

Given a finite graph $(V,E)$, a weight function on $(V,E)$ is a function $m: E \rightarrow \mathbb{R}_+$ and $(V,E)$ with a weight function is called a weighted graph. Given a weighted graph as above, we define for every $v \in V$, $m(v) = \sum_{e \in E, v \in e} m(e)$ and for a non-empty subset $U \subseteq V$, define $m(U) = \sum_{v \in U} m(v)$.   

The graph $(V,E)$ is called \textit{bipartite} if $V$ can be partitioned into two disjoint sets $V = S_1 \cup S_2$ called \textit{sides} such that for each $\lbrace u,v \rbrace \in E$, $\vert \lbrace u,v \rbrace \cap S_1 \vert =  \vert \lbrace u,v \rbrace \cap S_2 \vert =1$, i.e., each edge has exactly one vertex in each side.  We note that it follows that $m (S_1) = m(S_2) = \frac{1}{2} m(V)$.

We also define $\ell^2 (V,m)$ as in \cref{Vector valued subsec} above, i.e., $\ell^2 (V,m)$ is the space of functions $\phi: V \rightarrow \mathbb{C}$ with an inner-product
$$\langle \phi, \psi \rangle = \sum_{\lbrace v \rbrace \in V} = m(v) \phi (v) \overline{\psi (v)}.$$

The \textit{random walk} on $(V,E)$ as above is the operator $A: \ell^2 (V,m) \rightarrow \ell^2 (V,m)$ defined as  
$$(A \phi) (v) = \sum_{u \in V, \lbrace u, v \rbrace \in E} \frac{m (\lbrace u, v \rbrace)}{m(v)} \phi (u).$$
In the case where $m$ is constant $1$ on all the edges, then for every vertex $v$, $m(v)$ is the valency of $v$ and $A$ is called the \textit{simple random walk} on $(V,E)$. 

We state without proof a few basic facts regarding the random walk operator:
\begin{enumerate}
\item With the inner-product defined above, $A$ is a self-adjoint operator and the eigenvalues of $A$ lie in the interval $[-1,1]$.
\item The space of constant functions is an eigenspace of $A$ with eigenvalue $1$ and if $(V,E)$ is connected all other the other eigenfunctions of $A$ have eigenvalues strictly less than $1$.
\item The graph $(V,E)$ is bipartite if and only if $-1$ is an eigenvalue of $A$.
\end{enumerate}

Assuming that $(V,E)$ is bipartite with sides $S_1, S_2$,  we note that the spectrum of $A$ is symmetric: Explicitly, if $\lambda$ is an eigenvalue of $A$ with an eigenfunction $\phi$, then 
$$\phi ' (u) = 
\begin{cases}
\phi (u) & u \in S_1 \\
- \phi (u) & u \in S_2
\end{cases}$$
is an eigenfunction of $A$ with an eigenvalue $- \lambda$.  In particular $\phi = \mathbbm{1}_{S_1} - \mathbbm{1}_{S_2}$ is an eigenfunction with the eigenvalue $-1$.

We define the following averaging operators:
$M_{1}, M_{2} : \ell^2 (V, m) \rightarrow \mathbb{C}$: 
$$M_{i} \phi = \frac{1}{m (S_{i})} \sum_{u \in S_{i}} \phi (u).$$
We also define $M_{\sides} : \ell^2 (V, m) \rightarrow \ell^2 (V, m)$ by 
$$M_{\sides} \phi (u) = 
\begin{cases}
M_{1} \phi & u \in S_{1} \\
M_{2} \phi & u \in S_{2} 
\end{cases}.$$

We note that $M_{\sides}$ is the orthogonal projection  on the space of functions $\Span \lbrace \mathbbm{1}_V,   \mathbbm{1}_{S_i} - \mathbbm{1}_{S_j} \rbrace$.

We recall the following definition of spectral expansion:
\begin{definition}
Let $(V,E)$ be a finite connected graph with a weight function $m$ and $0 \leq \lambda <1$ be a constant. The graph $(V,E)$ is called a one-sided $\lambda$-spectral expander if the spectrum of $A$ is contained in $[-1, \lambda] \cup \lbrace 1 \rbrace$. 
\end{definition}
As noted above,  for a bipartite graph the spectrum of $A$ is symmetric with an eigenvalue $-1$ with an eigenfunction $\mathbbm{1}_{S_1} - \mathbbm{1}_{S_2} $.  Thus, for a connected bipartite graph $(V,E)$ it holds that the graph is a one-sided $\lambda$-spectral expander if and only if $\Vert A (I- M_{\sides}) \Vert \leq \lambda$. 

Given a Banach space $\B$, we consider the operator $(A (I-M_\sides)) \otimes \id_{\B} : \ell^2 (V,m ; \B) \rightarrow \ell^2 (V,m ; \B)$ and denote $\lambda_{(V,E),\bipartite}^\B = \Vert (A (I-M_\sides)) \otimes \id_{\B} \Vert_{B(\ell^2 (V,m ; \B))}$.

\begin{claim}
For every graph $(V,E)$ and every Banach space $\B$, $\lambda_{(V,E),\bipartite}^\B \leq 2$.
\end{claim}

\begin{proof}
By triangle inequality and linearity,
\begin{dmath*}
\Vert (A (I-M_\sides)) \otimes \id_{\B} \Vert_{B(\ell^2 (V,m ; \B))} \leq \\
\Vert A \otimes \id_{\B} \Vert_{B(\ell^2 (V,m ; \B))} + \Vert A \otimes \id_{\B}  \Vert_{B(\ell^2 (V,m ; \B))} \Vert M_\sides \otimes \id_{\B} \Vert_{B(\ell^2 (V,m ; \B))},
\end{dmath*}
and therefore in order to prove the claim, it is enough to show that 
$$\Vert A \otimes \id_{\B} \Vert_{B(\ell^2 (V,m ; \B))} \leq 1 \text{ and } \Vert M_\sides \otimes \id_{\B} \Vert_{B(\ell^2 (V,m ; \B))} \leq 1.$$
Indeed, by the convexity of the function $\vert . \vert^2$, for every $\phi \in \ell^2 (V,m ; \B)$, 
\begin{dmath*}
\Vert (A \otimes \id_{\B}) \phi \Vert^2 = \\
\sum_{v \in V} m(v) \left\vert \sum_{u \in V, \lbrace u, v \rbrace \in E} \frac{m (\lbrace u, v \rbrace)}{m(v)} \phi (u)\right\vert^2 \leq \\
\sum_{v \in V} m(v) \sum_{u \in V, \lbrace u, v \rbrace \in E} \frac{m (\lbrace u, v \rbrace)}{m(v)} \vert \phi (u) \vert^2 = \\
\sum_{u \in V} \vert \phi (u) \vert^2 \sum_{v \in V, \lbrace u, v \rbrace \in E} m (\lbrace u, v \rbrace) = \\
\sum_{u \in V} m(u) \vert \phi (u) \vert^2  = \Vert \phi \Vert^2,
\end{dmath*}
and
\begin{dmath*}
\Vert (M_\sides \otimes \id_{\B}) \phi \Vert^2 = \\
\sum_{v \in S_1} m(v) \left\vert \frac{1}{m (S_1)} \sum_{u \in S_i} m(u) \phi (u) \right\vert^2 +  \sum_{v \in S_2} m(v) \left\vert \frac{1}{m (S_2)} \sum_{u \in S_j} m(u) \phi (u) \right\vert^2\leq \\
\sum_{v \in S_1} m(v)  \frac{1}{m (S_1)} \sum_{u \in S_i} m(u) \vert \phi (u) \vert^2 +  \sum_{v \in S_2} m(v) \frac{1}{m (S_2)} \sum_{u \in S_2} m(u) \vert  \phi (u) \vert^2 = \\
\sum_{u \in S_1} m(u) \vert \phi (u) \vert^2 + \sum_{u \in S_2} m(u) \vert \phi (u) \vert^2 = \Vert \phi \Vert^2.
\end{dmath*}
\end{proof}

Combining this Claim with Lemma \ref{L contractive prop} and Corollary \ref{norm bound on theta-Hil coro} yields:
\begin{corollary}
\label{norm bound on rw on uc coro - bipartite}
Let $(V,E)$ be a connected finite graph with a weight function $m$ and  $0 < \lambda <1$ be a constant such that $(V,E)$ is a one-sided $\lambda$-spectral expander. For every monotone increasing function  $\omega : (0,1] \rightarrow (0,1]$ such that $\lim_{t \rightarrow 0^{+}} \omega (t) = 0$ and every $\B \in \mathcal{E}^{\text{u-curved}}_\omega$, we have that 
$$\lambda_{(V,E),\bipartite}^\B \leq 2 \omega (\lambda).$$

In particular, for every $0 < \theta \leq 1$ and every strictly $\theta$-Hilbertian space $\B$, we have that 
$$\lambda_{(V,E),\bipartite}^\B \leq 2 \lambda^\theta.$$
\end{corollary}

The expansion constant $\lambda_{(V,E),\bipartite}^\B$ can also be described as a cosine between projections as in Definition \ref{Angle between proj def}. In order to do this, we define $\ell^2 (E ; \B)$ to be the space of functions $\Phi : E \rightarrow \B$ with the norm
$$\Vert \Phi \Vert^2 = \sum_{\lbrace u, v \rbrace \in E} \vert \Phi (\lbrace u, v \rbrace ) \vert^2.$$
For a bipartite graph $(V,E)$ with sides $S_1, S_2$ we define the following projections on $\ell^2 (E ; \B)$: For $i=1,2$, and $\lbrace v_1, v_2 \rbrace \in E$ with $v_i \in S_i$ 
$$P_i \Phi (\lbrace v_1,v_2 \rbrace) = \frac{1}{m (v_i)} \sum_{\lbrace v_i, u \rbrace \in E} \Phi (  \lbrace v_i, u \rbrace ), \forall \Phi \in \ell^2 (E ; \B).$$
We also define 
$$P_{1,2} \Phi \equiv \frac{1}{\vert E \vert} \sum_{\lbrace u, v \rbrace \in E} \Phi (\lbrace u, v \rbrace), \forall \Phi \in \ell^2 (E ; \B).$$
Since $P_{1,2}$ is a projection on the space of constant functions, it follows that $P_i P_{1,2} = P_{1,2}$. One can also verify that $P_{1,2} P_i = P_{1,2}$ (we leave this to the reader).  Thus, we can define $\cos (\angle (P_1, P_2))$ as in Definition \ref{Angle between proj def}. 
\begin{proposition}
\label{spec gap equal cos prop}
For a finite connected bipartite graph $(V,E)$ and a Banach space $\B$ it holds that 
$$\cos (\angle (P_1, P_2)) = \lambda_{(V,E),\bipartite}^\B,$$
where $P_1,P_2$ are the projections defined above.
\end{proposition}

\begin{proof}
For $i=1,2$, define $\ell^2_i (V,m ; \B)$ to be the subspace of $\ell^2 (V,m ;B)$ composed of functions supported on $S_i$. Note every $\phi \in \ell^2 (V,m ;B)$ can be decomposed to $\phi = \phi_1 + \phi_2$ where $\phi_i \in \ell^2_i (V,m ; \B)$ by defining
$$\phi_i (v)= \begin{cases}
\phi (v) & v \in S_i \\
0 & v \notin S_i
\end{cases}.$$
Also note that for this decomposition 
$\Vert \phi \Vert^2 = \Vert \phi_1 \Vert^2 + \Vert \phi_2 \Vert^2$ and that 
$$\Vert ((A (I- M_{\sides})) \otimes \id_{\B}) \phi \Vert^2 = \Vert ((A (I- M_{\sides}))\otimes \id_{\B}) \phi_1 \Vert^2 + \Vert ((A (I- M_{\sides}))\otimes \id_{\B}) \phi_2 \Vert^2.$$
Thus for every non-zero $\phi$,
$$\frac{\Vert ((A (I- M_{\sides}))\otimes \id_{\B}) \phi \Vert^2}{\Vert \phi \Vert^2} = \frac{\Vert ((A (I- M_{\sides}))\otimes \id_{\B}) \phi_1 \Vert^2 + \Vert ((A (I- M_{\sides})) \otimes \id_{\B}) \phi_2 \Vert^2}{\Vert \phi_1 \Vert^2 + \Vert \phi_2 \Vert^2}.$$
It follows that 
$$\Vert (A (I- M_{\sides})) \otimes \id_{\B} \Vert = \max_{i=1,2} \Vert \left.  (A (I- M_{\sides})) \otimes \id_{\B} \right\vert_{\ell^2_i (V, m ; \B)} \Vert.$$

Define $L_i : \ell^2_i (V, m ; \B) \rightarrow \ell^2 (E ; \B)$ by 
$$L_i \phi (\lbrace v_1, v_2 \rbrace) = \phi (v_i),$$
for all $\lbrace v_1, v_2 \rbrace \in E$ such that $v_1 \in S_1, v_2 \in S_2$.  Note that $L_i$ is an isometry onto $\im (P_i) \subseteq \ell^2 (E ; \B)$.  Observe that by the definition of all the operators it follows that
\begin{dmath*}
\left. (A (I- M_{\sides}))\otimes \id_{\B} \right\vert_{\ell^2_1 (V, m ; \B)} =L_2^{-1} P_2 (I - P_{1,2} ) L_1 = L_2^{-1} P_2 (I - P_{1,2} ) P_1 L_1  = L_2^{-1} (P_2 P_1  - P_{1,2} ) P_1 L_1 ,
\end{dmath*}
and
$$\left.  (A (I- M_{\sides})) \otimes \id_{\B} \right\vert_{\ell^2_2 (V, m ; \B)} =L_1^{-1} (P_1 P_2 - P_{1,2} ) L_2.$$
Thus, using the fact that $L_i$ are isometries onto $\im (P_i)$ it follows that 
$$\Vert \left. (A (I- M_{\sides})) \otimes \id_{\B} \right\vert_{\ell^2_1 (V, m ; \B)} \Vert = \Vert P_2 P_1  - P_{1,2} \Vert,$$
and
$$\Vert \left. (A (I- M_{\sides})) \otimes \id_{\B} \right\vert_{\ell^2_2 (V, m ; \B)} \Vert = \Vert P_1 P_2  - P_{1,2} \Vert,$$
as needed.

\end{proof}

\subsection{Weighted simplicial complexes}

Given a set $V$, an abstract simplicial complex $X$ with a vertex set $V$ is a family of subsets $X \subseteq 2^V$ such that if $\tau \in X$ and $\eta \subseteq \tau$, then $\eta \in X$. We will denote $X(k)$ to be the sets in $X$ of cardinality $k+1$.  A simplicial complex $X$ is called \textit{$n$-dimensional} if $X(n+1) = \emptyset$. An $n$-dimensional simplicial complex $X$ is called \textit{pure} $n$-dimensional if for every $\tau$ in $X$, there is $\sigma \in X(n)$ such that $\tau \subseteq \sigma$.  

A pure $n$-dimensional simplicial complex $X$ is called \textit{gallery connected} if for every $\sigma, \sigma ' \in X(n)$, there is a finite sequence $\sigma_1,...,\sigma_l \in X(n)$ such that $\sigma = \sigma_1, \sigma ' = \sigma_l$ and for every $1 \leq i \leq l-1$, $\sigma_i \cap \sigma_{i+1} \in X(n-1)$. Below, we will always assume that $X$ is pure $n$-dimensional and gallery connected.


We define the following weight function $m: \bigcup_{k=0}^{n} \X (k) \rightarrow \mathbb{R} \cup \lbrace \infty \rbrace$ inductively as follows: 
$$\forall \sigma \in \X (n), m(\sigma) =1,$$
For $0 \leq k \leq n-1$ and $\tau \in X(k)$,
$$m (\tau) = \sum_{\sigma \in \X (k+1), \tau \subseteq \sigma} m(\sigma).$$
More explicitly,
$$\forall \tau \in \X (k), m(\tau) = (n-k)! \vert \lbrace \sigma \in \X (n) : \tau \subseteq \sigma  \rbrace \vert.$$

Given a simplex $\tau \in X$, the \textit{link of $\tau$} is the subcomplex of $X$, denoted $X_\tau$, that is defined as
$$X_\tau = \lbrace \eta \in X : \tau \cap \eta = \emptyset, \tau \cup \eta \in X \rbrace.$$
We define a weight function $m_\tau$ on $X_\tau$ by 
$$m_\tau (\eta) = m (\tau \cup \eta).$$
Below we will only be interested in the $1$-dimensional links of $X$: We note that if $\tau \in X(k)$ and $X$ is pure $n$-dimensional, then $X_\tau$ is pure $(n-k-1)$-dimensional. In particular, if $\tau \in X(n-2)$, then $X_\tau$ is a graph. Thus, we will refer to all links of the form $X_\tau$ where $\tau \in X(n-2)$ as \textit{the $1$-dimensional links of $X$}.  For $1$-dimensional links, the weight function is easy to describe: For a $1$-dimensional link $X_\tau$, we note that this weight function assigns the weight $1$ for each edge in the link $X_\tau$ and for each every $\lbrace u  \rbrace \in X_\tau (0)$, $m_\tau (\lbrace u \rbrace)$ is the valency of $\lbrace u \rbrace$ in $X_\tau$.

Given a pure $n$-dimensional simplicial complex $X$, we call $X$ \textit{partite}  if there is a partition of the vertex set $V$,  $V_0 \sqcup ...  \sqcup V_n = V$ such that for every $\sigma \in X(n)$ and every $0 \leq i \leq n$ it holds that $\vert \sigma \cap V_i \vert =1$, i.e., every $n$-dimensional simplex has exactly one vertex in each of the sets $V_0,...,V_n$.  A partite simplicial complex is also sometimes called \textit{colorable}, since we can think of the partition $V_0,...,V_n$ as a coloring of the vertex sets with $n+1$ colors, such that each $n$-dimensional simplex has vertices with all the colors (or equivalently, each $n$-dimensional simplex do not have two vertices with the same color).  For an $n$-dimensional partite simplicial complex, we define a type function $\type : X \rightarrow 2^{\lbrace 0,..., n\rbrace}$ by 
$$\type (\tau) = \lbrace i : \exists v \in \tau \cap V_i \rbrace.$$

\section{\.Zuk's criterion for partite links}

\label{Zuk criterion sec}

Let $X$ be a partite pure $n$-dimensional simplicial complex such that $X$ is gallery connected and the $1$-dimensional links of $X$ are connected finite graphs.  Also let $G$ be a locally compact, unimodular group acting on $X$ such that the action is cocompact and for every $\tau \in X(n-2) \cup X(n-1) \cup X(n)$ the subgroup stabilizing $\tau$, denoted $G_\tau$, is an open compact subgroup.  We also assume that the action of $G$ is type-preserving, i.e., for every $0 \leq k \leq n$ and every $\tau \in X(k)$, it holds for every $g \in G$ that $\type (\tau) = \type (g.\tau)$.  Last, let $\pi$ be a continuous representation $\pi$ of $G$ on a Banach space $\B$. 

Note that by the assumption that the action of $G$ is type preserving it follows for each $\sigma \in X(k)$ that if $g. \sigma = \sigma$, then for each vertex $v \in \sigma$ it holds that $g. v = v$.   For every $n-2 \leq k \leq n$,  we choose a fundamental domain $D(k)$ of the action of $G$ on $X(k)$ such that the following holds: for every $0 \leq k_1 < k_2 \leq n$ and every $\tau \in D(k_1)$ there is $\sigma \in D (k_2)$ such that $\tau \subseteq \sigma$. 

Define $C (X(n), \pi)$ to be the space of maps $\phi : X(n) \rightarrow \B$ that are equivariant with respect to $\pi$, i.e., for every $\sigma \in X(n)$ and every $g \in G$, $\pi (g) \phi (\sigma) = \phi (g. \sigma)$.  Define a norm on $C (X(n), \pi)$ by
$$\Vert \phi \Vert^2 = \sum_{\sigma \in D(n)} \frac{1}{\mu (G_\sigma )} \vert \phi (\sigma) \vert^2,$$
where $G_\sigma$ is the subgroup stabilizing $\sigma$ and $\mu$ is the Haar measure of $G$.  

Define $K \subseteq G$ to be the set
$$K = \lbrace g \in G : \exists \sigma, \sigma ' \in D(n), \vert g. \sigma \cap \sigma ' \vert \geq n-1 \rbrace.$$
Note that $K$ is a compact symmetric set that generates $G$ (the last fact is due to the assumption that $X$ is gallery connected). 

The following Lemma is a variant of \cite[Lemma 3.3]{DJ1}:
\begin{lemma}
\label{change of summation lemma}
Let $n-2 \leq k < l \leq n$ and 
$$\psi :  \lbrace (\tau, \eta ) \in X(k) \times X(l) : \tau \subseteq \eta \rbrace \rightarrow [0, \infty) .$$
Assume that there is a constant $\kappa \geq 1$ such that for every $g \in K$ and every $ (\tau, \eta) \in  X(k) \times X(l), \tau \subseteq \eta$ it holds that $\psi (g.  \tau, g. \eta) \leq \kappa \psi (\tau, \eta)$. 
 Then 
$$\sum_{\eta \in D(l)} \frac{1}{\mu (G_\eta)} \sum_{\tau \in X(k), \tau \subseteq \eta} \psi (\tau, \eta) \leq \kappa^2 \sum_{\tau \in D(k)} \frac{1}{\mu (G_\tau)} \sum_{\eta \in X(k), \tau \subseteq \eta} \psi (\tau, \eta)$$
and
$$\sum_{\eta \in D(l)} \frac{1}{\mu (G_\eta)} \sum_{\tau \in X(k), \tau \subseteq \eta} \psi (\tau, \eta) \geq \frac{1}{\kappa^2} \sum_{\tau \in D(k)} \frac{1}{\mu (G_\tau)} \sum_{\eta \in X(k), \tau \subseteq \eta} \psi (\tau, \eta).$$
\end{lemma}

\begin{proof}
Note that since $K$ is symmetric, it follows that $\psi (g.  \tau, g. \eta) \geq \frac{1}{\kappa} \psi (\tau, \eta)$ for every $g \in K$ and every $ (\tau, \eta) \in  X(k) \times X(l), \tau \subseteq \eta$. Thus the proofs of both inequalities are similar (only using reverse inequalities) and we will only prove the first one. 

For every $\tau \in D(k), \eta \in D(l)$,  we fix a maximal set $A_{(\tau, \eta)} \subseteq G$ such that $\forall g \in A_{(\tau, \eta)}$, $g.\tau \subseteq \eta$ and for $g, g' \in A_{(\tau, \eta)}$, if $g \neq g'$, then $g.\tau \neq g' .\tau$ (it may be that set $A_{(\tau, \eta)}$ is the empty set).     Note that by the choice of $D(k), D(l)$, it follows that $A_{(\tau, \eta)} \subseteq K$.  Thus 


\begin{dmath*}
\sum_{\eta \in D(l)} \frac{1}{\mu (G_\eta)} \sum_{\tau \in X(k), \tau \subseteq \eta} \psi (\tau, \eta)    =
\sum_{\eta \in D(l)} \frac{1}{\mu (G_\eta)} \sum_{\tau \in D(k)} \sum_{g \in A_{(\tau, \eta)}} \psi (g.\tau, \eta)  \leq
\kappa \sum_{\eta \in D(l), \tau \in D(k)} \frac{1}{\mu (G_\eta)} \sum_{g \in A_{(\tau, \eta)}} \psi (\tau, g^{-1}. \eta) \leq
\kappa^2 \sum_{\eta \in D(l), \tau \in D(k)} \frac{1}{\mu (G_\eta)} \sum_{g \in A_{(\tau, \eta)}} \frac{1}{\mu (G_\tau)} \int_{G_\tau} \psi (\tau, g' g^{-1}. \eta) d \mu (g') =  
\kappa^2 \sum_{\tau \in D(k)} \frac{1}{\mu (G_\tau)}  \sum_{\eta \in D(l)} \frac{1}{\mu (G_\eta)} \int_{g'', \tau \subseteq 
g''.\eta} \psi (\tau, g''. \eta) d \mu (g'') =  
\kappa^2 \sum_{\tau \in D(k)} \frac{1}{\mu (G_\tau)}  \sum_{\eta \in X(l), \tau \subseteq \eta} \psi (\tau, \eta).
\end{dmath*}
\end{proof}

\begin{corollary}
\label{change of summation coro}
Let $\nu \subseteq \lbrace 0,...,n\rbrace$ such that $\vert \nu \vert = n-1$ or $\vert \nu \vert = n$. For every $\phi \in C(X(n), \pi)$ it holds that
$$\Vert \phi \Vert^2 \leq \left( \sup_{g \in K} \Vert \pi (g) \Vert^2 \right) \sum_{\tau \in D (\vert \nu \vert -1), \type (\tau) = \nu} \frac{1}{\mu (G_\tau)}  \sum_{\sigma \in X(n), \tau \subseteq \eta} \vert \phi (\sigma) \vert^2,$$
and
$$\Vert \phi \Vert^2 \geq \left( \sup_{g \in K} \Vert \pi (g) \Vert^2 \right)^{-1} \sum_{\tau \in D (\vert \nu \vert -1), \type (\tau) = \nu} \frac{1}{\mu (G_\tau)}  \sum_{\sigma \in X(n), \tau \subseteq \eta} \vert \phi (\sigma) \vert^2,$$
\end{corollary}

\begin{proof}
The proofs of both the inequalities are similar and we will prove only the first one.  Let $\phi \in C (X(n), \pi)$ and $\nu$ as above. Denote $k =\vert \nu \vert -1$ and define
$$\psi : \lbrace (\tau, \sigma ) \in X(k) \times X(n) : \tau \subseteq \sigma \rbrace \rightarrow [0, \infty)$$
as 
$$\psi (\tau, \sigma ) = \begin{cases}
\vert \psi (\sigma ) \vert^2 & \type (\tau) = \nu \\
0 & \type (\tau) \neq \nu
\end{cases},$$
for all $(\tau, \sigma ) \in X(k) \times X(n)$ with $\tau \subseteq \sigma$. Note that for every $g \in K$ it holds that 
$$\psi (g.\tau, g. \sigma ) \leq  \left( \sup_{g \in K} \Vert \pi (g) \Vert \right) \psi (\tau, \sigma )$$
(recall that the action is type preserving). Thus, by Lemma \ref{change of summation lemma} it follows that 
\begin{dmath*}
\Vert \phi \Vert^2 = 
\sum_{\sigma \in D(n)} \frac{1}{\mu (G_\sigma)} \vert \phi (\sigma )\vert^2 =
\sum_{\sigma \in D(n)} \frac{1}{\mu (G_\sigma)}\sum_{\tau \in X(k), \tau \subseteq \sigma} \psi (\tau, \sigma) \leq \\
\left( \sup_{g \in K} \Vert \pi (g) \Vert^2 \right) \sum_{\tau \in D(k)} \frac{1}{\mu (G_\tau)}\sum_{\sigma \in X(n), \tau \subseteq \sigma} \psi (\tau, \sigma) = \\
\left( \sup_{g \in K} \Vert \pi (g) \Vert^2 \right) \sum_{\tau \in D(k), \type (\tau) = \nu} \frac{1}{\mu (G_\tau)}\sum_{\sigma \in X(n), \tau \subseteq \sigma} \vert \phi (\sigma) \vert^2,
\end{dmath*}
as needed.
\end{proof}

For $x \in \B$, define $\phi_x : D(n) \rightarrow \B$ by 
$$\phi_x (\sigma) =\pi \left( \frac{ \mathbbm{1}_{G_{\sigma}}}{\mu (G_\sigma )} \right) x,  \forall \sigma \in D(n),$$
where $\mathbbm{1}_{G_{\sigma}}$ is the indicator function of $G_\sigma$.  Observe that $\phi_x (\sigma) \in \B^{\pi (G_\sigma)}$ for every $\sigma \in D(n)$.  Extend $\phi_x$ to $X(n)$ as follows: for every $g \in G$ and every $\sigma \in D(n)$, define 
$$\phi_x (g.\sigma) = \pi (g) \phi_x (\sigma) =   \pi \left( \frac{ \mathbbm{1}_{g^{-1} G_{\sigma}}}{\mu (g^{-1} G_\sigma )} \right) x =  \pi \left( \frac{ \mathbbm{1}_{g^{-1} G_{\sigma}}}{\mu (G_\sigma )} \right) x.$$

\begin{proposition}
The map $\phi_x$ is well-defined and equivariant, i.e., $\phi_x \in C(X(n), \pi)$.
\end{proposition} 

\begin{proof}
To show that $\phi_x$ is well-defined, we need to show that for every $\sigma \in D(n)$ and every $g_1, g_2 \in G$, if $g_1 . \sigma = g_2 . \sigma$, then $\pi (g_1) \phi_x (\sigma) = \pi (g_2) \phi_x (\sigma)$ or equivalently, that 
$\pi (g_2^{-1} g_1) \phi_x (\sigma) = \phi_x (\sigma)$. This equality follows from the fact that $g_2^{-1} g_1 \in G_\sigma$ and that $\phi_x (\sigma) \in \B^{\pi (G_\sigma)}$. 

The fact that $\phi_x$ is equivariant readily follows from its definition.
\end{proof}

Define linear maps $R_{1}, R_2$ as follows:
$$R_1 : \B \rightarrow C(X(n), \pi), R_1 x = \phi_x,$$
$$R_2 : C(X(n), \pi) \rightarrow \B, R_2 \phi = \frac{1}{\vert D(n) \vert} \sum_{\sigma \in D(n)} \phi (\sigma).$$

\begin{proposition}
\label{bound R_1, R_2 prop}
The maps $R_1, R_2$ are bounded and 
$$\Vert R_1 \Vert \leq (\max_{\sigma \in D(n)} \sup_{g \in G_\sigma} \Vert \pi (g) \Vert ) \sqrt{\sum_{\sigma \in D(n)} \frac{1}{\mu (G_\sigma )}},$$
$$\Vert R_2 \Vert \leq \max_{\sigma \in D(n)} \sqrt{\mu (G_\sigma )}.$$
\end{proposition}

\begin{proof}
Note that for every $x \in \B$ with $\vert x \vert =1$ and every $\sigma \in D(n)$, 
\begin{dmath*}
\vert \phi_x (\sigma) \vert = \left\vert \frac{1}{\mu (G_\sigma) } \int_{G_{\sigma}} \pi ( g) x d \mu (g) \right\vert \leq 
\frac{1}{\mu (G_\sigma )} \int_{G_{\sigma}} \vert \pi ( g) x \vert d \mu (g) \leq 
\sup_{g \in G_\sigma} \Vert \pi (g) \Vert. 
\end{dmath*}
Thus for $x \in \B$, with $\vert x \vert =1$, it follows that 
$$\Vert \phi_x \Vert^2 \leq  \sum_{\sigma \in D(n)} \frac{1}{\mu (G_\sigma )} (\sup_{g \in G_\sigma} \Vert \pi (g) \Vert)^2,$$
and the bound on the norm of $R_1$ follows. 

Next, let $\phi \in C(X(n), \pi)$ with $\Vert \phi \Vert =1$. Then for every $\sigma \in D(n)$, $\vert \phi (\sigma) \vert^2 \leq \mu (G_\sigma )$ and it follows that 
\begin{dmath*}
\left\vert R_2 \phi \right\vert \leq \frac{1}{\vert D(n) \vert} \sum_{\sigma \in D(n)} \sqrt{ \mu (G_\sigma )} \leq \max_{\sigma \in D(n)} \sqrt{\mu (G_\sigma )}.
\end{dmath*}
\end{proof}

For $\emptyset \neq \nu \subseteq \lbrace 0,....,n \rbrace$ define an equivalence relation $\sim_\nu$ on $X(n)$ as $\sigma \sim_\nu \sigma '$ if $\vert \sigma \cap \sigma ' \vert \geq \vert \nu \vert$ and there is $\tau \in X$ such that $\type (\tau) = \nu$ and $\tau \subseteq \sigma \cap \sigma ' $.  Note that by the assumption that the $1$-dimensional links of $X$ are finite graphs, it follows that for every $\emptyset \neq \nu \subseteq \lbrace 0, ...,n\rbrace$, $\vert \nu \vert \geq n-1$ and every $\sigma \in X(n)$,  the set $\lbrace \sigma ' \in X(n) : \sigma \sim_\nu \sigma ' \rbrace$ is finite. For $\emptyset \neq \nu \subseteq \lbrace 0, ...,n\rbrace$, $\vert \nu \vert \geq n-1$,  define a projection $P_\nu^\pi : C (X(n), \pi) \rightarrow C (X(n), \pi)$ by 
$$P_\nu^\pi \phi (\sigma) = \frac{1}{\vert \lbrace \sigma ' \in X(n) : \sigma \sim_\nu \sigma ' \rbrace \vert} \sum_{\sigma ' \sim_\nu \sigma} \phi (\sigma '),$$
(verifying that $P_\nu^\pi \phi$ is equivariant with respect to $\pi$ is straight-forward and left for the reader).  Denote 
$$C (X(n), \pi)_\nu = \lbrace \phi \in C( X(n), \pi) : \forall \sigma, \sigma ', \sigma \sim_\nu \sigma ' \Rightarrow \phi (\sigma ) = \phi (\sigma ') \rbrace,$$
and note that $P_\nu^\pi$ is a projection on $ C (X(n), \pi)_\nu$.

\begin{lemma}
\label{intersection of images is const lemma}
The space $\bigcap_{\nu \subseteq \lbrace 0,....,n \rbrace, \vert \nu \vert =n} \im (P_\nu^\pi)$ is the space of all the constant maps $\phi \equiv x_0$ where $x_0$ is a fixed point of the action of $G$ on $\B$.
\end{lemma}

\begin{proof}
Fix $\phi \in \bigcap_{\nu \subseteq \lbrace 0,....,n \rbrace, \vert \nu \vert =n} \im (P_\nu^\pi)$.  As noted above $\im (P_\nu^\pi) = C (X(n), \pi)_\nu$ and it follows that if $\sigma \cap \sigma ' \in X(n-1)$, then $\phi (\sigma) = \phi (\sigma ')$.  Thus if $\sigma, \sigma ' \in X(n)$ are connected by a gallery, then $\phi (\sigma) = \phi (\sigma ')$. By assumption, $X$ is a gallery connected and thus $\phi$ is a constant map, i.e., there is $x_0 \in \B$ such that $\phi \equiv x_0$. Since $\phi$ is equivariant it follows that $x_0$ is a fixed point of the action of $G$ on $\B$.
\end{proof}

\begin{proposition}
For every $\nu, \nu ' \subseteq \lbrace 0,...,n \rbrace$ such that $\vert \nu \vert = \vert \nu ' \vert =n$ and $\nu \neq \nu '$ it holds that 
$$\im (P_\nu^\pi) \cap \im (P_{\nu '}^\pi) = \im (P_{\nu \cap \nu '}^\pi)$$
and
$$P_{\nu \cap \nu '}^\pi P_{\nu}^\pi = P_{\nu}^\pi P_{\nu \cap \nu '}^\pi   = P_{\nu \cap \nu '}^\pi.$$
\end{proposition}

\begin{proof}
Note that since $\nu \cap \nu ' \subseteq \nu, \nu '$, it follows that $C(X(n), \pi)_{ \nu \cap \nu '} \subseteq C(X(n), \pi)_{ \nu} ,  C(X(n), \pi)_{ \nu '}$. Thus, 
$$\im (P_\nu^\pi) \cap \im (P_{\nu '}^\pi) \supseteq \im (P_{\nu \cap \nu '}^\pi).$$
For the reverse inclusion, note for every $\tau \in X(n-2)$ with $\type (\tau) = \nu \cap \nu ' $, the link $X_\tau$ is connected and thus we can argue as in the proof of Lemma \ref{intersection of images is const lemma} above for the set $\lbrace \sigma : \tau \subseteq \sigma \rbrace$ and deduce that if  $\phi \in \im (P_\nu^\pi) \cap \im (P_{\nu '}^\pi)$, then $\phi$ is constant on $\tau$. This is true for every $\tau \in X(n-2)$ with $\type (\tau) = \nu \cap \nu ' $ and thus
$$\im (P_\nu^\pi) \cap \im (P_{\nu '}^\pi) \subseteq \im (P_{\nu \cap \nu '}^\pi).$$

Second, since $P_{\nu \cap \nu '}^\pi, P_{\nu}^\pi$ are projections and  $\im (P_{\nu \cap \nu '}^\pi) \subseteq \im (P_\nu^\pi)$, it follows that $P_{\nu}^\pi P_{\nu \cap \nu '}^\pi   = P_{\nu \cap \nu '}^\pi$. 

Last, note that for every two $\sigma, \sigma ' \in X(n)$, if $\sigma \sim_{\nu} \sigma '$, then $\sigma \sim_{\nu \cap \nu '} \sigma '$.  Thus, for every $\phi \in C(X(n), \pi)$ and every $\sigma \in X(n)$ it holds that
\begin{dmath*}
{P_{\nu \cap \nu '}^\pi  P_\nu^\pi \phi (\sigma) =  } \\
 \frac{1}{\vert \lbrace \sigma ' \in X(n) : \sigma \sim_{\nu \cap \nu '} \sigma ' \rbrace \vert} \sum_{\sigma ' \sim_{\nu \cap \nu '} \sigma}   \frac{1}{\vert \lbrace \sigma '' \in X(n) : \sigma ' \sim_\nu \sigma '' \rbrace \vert} \sum_{\sigma '' \sim_\nu \sigma '} \phi (\sigma '') = \\
  \frac{1}{\vert \lbrace \sigma ' \in X(n) : \sigma \sim_{\nu \cap \nu '} \sigma ' \rbrace \vert} \sum_{\sigma '' \sim_{\nu \cap \nu '} \sigma}  \phi (\sigma '') \sum_{\sigma ' \sim_\nu \sigma ''} \frac{1}{\vert \lbrace \sigma '' \in X(n) : \sigma ' \sim_\nu \sigma '' \rbrace \vert}  =   P_{\nu \cap \nu '}^\pi \phi (\sigma).
\end{dmath*}
Therefore $P_{\nu \cap \nu '}^\pi  P_\nu^\pi = P_{\nu \cap \nu '}^\pi$ as needed.
\end{proof}

The above Proposition implies that we can define $\cos (\angle (P_\nu^\pi, P_{\nu '}^\pi))= \cos_{P_{\nu \cap \nu '}^\pi} (\angle (P_\nu^\pi, P_{\nu '}^\pi))$ for every $\nu , \nu ' \subseteq \lbrace 0,...,n \rbrace, \nu \neq \nu '$ with $\vert \nu \vert = \vert \nu ' \vert = n$. 

\begin{lemma}
\label{projection norm bound lemma}
For every $\nu \subseteq \lbrace 0,..., n \rbrace$, $\vert \nu \vert = n$ it holds that 
$$\Vert P_\nu^\pi \Vert \leq \sup_{g \in K} \Vert \pi (g) \Vert^2  .$$
\end{lemma}

\begin{proof}
Fix $\nu \subseteq \lbrace 0,..., n \rbrace$, $\vert \nu \vert = n$ and fix some $\phi \in C(X(n), \pi)$.  

For every $\tau \in D(n-1)$ with $\type (\tau) = \nu$,  it follows by the convexity of the norm that 
\begin{dmath*}
\sum_{\sigma \in X(n), \tau \subseteq \sigma} \vert  P_\nu^\pi \phi (\sigma) \vert^2 \leq 
\sum_{\sigma \in X(n), \tau \subseteq \sigma} \frac{1}{\vert \lbrace \sigma ' \in X(n) : \sigma \sim_{\nu} \sigma ' \rbrace \vert} \sum_{\sigma ' \in X(n),  \sigma \sim_{\nu} \sigma '} \vert \phi (\sigma ') \vert^2 = \\
\sum_{\sigma \in X(n), \tau \subseteq \sigma} \frac{1}{\vert \lbrace \sigma ' \in X(n) : \tau \subseteq \sigma '  \rbrace \vert} \sum_{\sigma ' \in X(n),  \tau \subseteq \sigma '} \vert \phi (\sigma ') \vert^2 = \\
\sum_{\sigma ' \in X(n), \tau \subseteq \sigma '} \vert \phi (\sigma ') \vert^2  \sum_{\sigma  \in X(n),  \tau \subseteq \sigma}  \frac{1}{\vert \lbrace \sigma  \in X(n) : \tau \subseteq \sigma  \rbrace \vert} = \\
\sum_{\sigma ' \in X(n), \tau \subseteq \sigma '} \vert \phi (\sigma ') \vert^2.
\end{dmath*}

By Corollary \ref{change of summation coro} and the previous inequality, it holds that 
\begin{dmath*}
\Vert P_\nu^\pi \phi \Vert^2 \leq \left( \sup_{g \in K} \Vert \pi (g) \Vert^2 \right) \sum_{\tau \in D(n-1), \type (\tau) =\nu} \frac{1}{\mu (G_
\tau)} \sum_{\sigma \in X(n), \tau \subseteq \sigma} \vert  P_\nu^\pi \phi (\sigma) \vert^2 \leq 
\left( \sup_{g \in K} \Vert \pi (g) \Vert^2 \right) \sum_{\tau \in D(n-1), \type (\tau) =\nu} \frac{1}{\mu (G_
\tau)}\sum_{\sigma ' \in X(n), \tau \subseteq \sigma '} \vert \phi (\sigma ') \vert^2 \leq
\left( \sup_{g \in K} \Vert \pi (g) \Vert^4 \right) \Vert \phi \Vert^2
\end{dmath*}
as needed.
\end{proof}

We note that for every $\tau \in X(n-2)$, $X_\tau$ is a bipartite graph and we denote $\lambda_{\tau,   \bipartite}^\B = \lambda_{X_\tau,   \bipartite}^\B$. 
\begin{lemma}
\label{angle bound by bipartite link lemma}
For every $\nu, \nu ' \subseteq \lbrace 0,...,n \rbrace$ such that $\vert \nu \vert = \vert \nu ' \vert =n$ and $\nu \neq \nu '$, denote 
$$ \lambda_{\nu \cap \nu '}^\B =  \max_{\tau \in D (n-2), \type (\tau) = \nu \cap \nu '} \lambda_{\tau,   \bipartite}^\B .$$
Then it holds that 
$$\cos (\angle (P_\nu^\pi, P_{\nu '}^\pi)) \leq  \left(\sup_{g \in K} \Vert \pi (g) \Vert^2 \right) \lambda_{\nu \cap \nu '}^\B.$$ 
\end{lemma}

\begin{proof}

Let $\phi \in C (X(n), \pi)$ be some map.  Without loss of generality, it is enough to show that 
$$\Vert  (P_\nu^\pi P_{\nu '}^\pi - P_{\nu \cap \nu '}^\pi ) \phi \Vert \leq \left(\sup_{g \in K} \Vert \pi (g) \Vert^2 \right) \lambda_{\nu \cap \nu '}^\B \Vert \phi \Vert.$$

For every $\tau \in D(n-2)$ with $\type (\tau) = \nu \cap \nu ' = \lbrace 2,...,n \rbrace$ the link $X_\tau$ is a bipartite graph with sides 
$$S_1 = \lbrace v \in X_\tau (0) : \type (\lbrace v \rbrace) = \lbrace 0,..., n \rbrace \setminus \nu \rbrace,$$
$$S_2 = \lbrace v \in X_\tau (0)  : \type (\lbrace v \rbrace) = \lbrace 0,..., n \rbrace \setminus \nu '  \rbrace.$$
Let $P_1, P_2, P_{1,2} : \ell^2 (X_\tau (1) ; \B) \rightarrow  \ell^2 (X_\tau (1) ; \B)$ defined in  \cref{Random walks on bipartite finite graphs subsec}, i.e., 
$$P_i \Phi (\lbrace v_1,v_2 \rbrace) = \frac{1}{m (v_i)} \sum_{\lbrace v_i, u \rbrace \in E} \Phi (  \lbrace v_i, u \rbrace ), \forall \Phi \in \ell^2 (X_\tau (1) ; \B),$$
and
$$P_{1,2} \Phi \equiv \frac{1}{\vert X_\tau (1) \vert} \sum_{\lbrace u, v \rbrace \in X_\tau (1)} \Phi (\lbrace u, v \rbrace), \forall \Phi \in \ell^2 (X_\tau (1) ; \B).$$
Define the localization of $\phi$ on $X_\tau$ to be the function $\phi_{\tau} \in \ell^2 (X_\tau (1) ; \B)$ defined as 
$$\phi_\tau (\lbrace u,v \rbrace) = \phi (\tau \cup  \lbrace u,v \rbrace).$$
One can verify that for every $\sigma \in X(n)$ with $\tau \subseteq \sigma$ it holds that 
$$P_\nu \phi (\sigma) = P_1 \phi_\tau (\sigma \setminus \tau), P_{\nu '} \phi (\sigma) = P_2 \phi_\tau (\sigma \setminus \tau) , P_{\nu \cap \nu '} \phi (\sigma) = P_{1,2} \phi_\tau (\sigma \setminus \tau).$$ 
Let $\tau \in D(n-2)$ with $\type (\tau) = \nu \cap \nu '$.  By Proposition \ref{spec gap equal cos prop} it follows that
\begin{dmath*}
\sum_{\sigma \in X(n), \tau \subseteq \sigma} \vert (P_\nu^\pi P_{\nu '}^\pi - P_{\nu \cap \nu '}^\pi ) \phi (\sigma) \vert^2 = \\
\sum_{\lbrace u,v \rbrace \in X_\tau (1)} \vert (P_1 P_{2} - P_{1,2} ) \phi_\tau (\lbrace u,v \rbrace) \vert^2 \leq  \\
\lambda_{\tau,   \bipartite}^\B \sum_{\lbrace u,v \rbrace \in X_\tau (1)} \vert \phi_\tau (\lbrace u,v \rbrace) \vert^2 \leq \\ 
\lambda_{\tau,   \bipartite}^\B \sum_{\sigma \in X(n), \tau \subseteq \sigma} \vert \phi (\sigma) \vert^2 \leq 
\lambda_{\nu \cap \nu '}^\B \sum_{\sigma \in X(n), \tau \subseteq \sigma} \vert \phi (\sigma) \vert^2.
\end{dmath*}

By Corollary \ref{change of summation coro} (applied twice) and the above computation,  
\begin{dmath*}
\Vert (P_\nu^\pi P_{\nu '}^\pi - P_{\nu \cap \nu '}^\pi ) \phi \Vert^2 \leq  \\
{\left(\sup_{g \in K} \Vert \pi (g) \Vert^2 \right) \sum_{\tau \in D(n-2), \type (\tau) = \nu \cap \nu'} \frac{1}{\mu (G_\tau)} \sum_{\sigma \in X(n), \tau \subseteq \sigma} \vert (P_\nu^\pi P_{\nu '}^\pi - P_{\nu \cap \nu '}^\pi ) \phi (\sigma) \vert^2 } \leq \\
{\left(\sup_{g \in K} \Vert \pi (g) \Vert^2 \right) \sum_{\tau \in D(n-2), \type (\tau) = \nu \cap \nu'} \frac{1}{\mu (G_\tau)}  \lambda_{\nu \cap \nu '}^\B \sum_{\sigma \in X(n), \tau \subseteq \sigma} \vert \phi (\sigma) \vert^2 } \leq 
\left(\sup_{g \in K} \Vert \pi (g) \Vert^4 \right) \lambda_{\nu \cap \nu '}^\B \Vert \phi \Vert^2,
\end{dmath*}
as needed.
\end{proof}

\begin{lemma}
\label{f_k lemma}
Let $\pi$ be a representation of $G$ and define $T_\pi = \frac{1}{n+1} \sum_{\nu \subseteq \lbrace 0,...,n \rbrace,\vert \nu \vert  =n} P_\nu^\pi$. For every $k \in \mathbb{N}$, there is a real valued positive function $f_k \in C_c (G)$ such that $\int_G f_k d \mu = 1$ (i.e., $f_k$ is a compactly support probability function) and 
$$R_2 T_\pi^k R_1 = \pi (f_k).$$
\end{lemma}

\begin{proof}
We start by noting that for every $\sigma \in D(n)$ and every $k \in \mathbb{N}$,  there is a probability function $f^k_\sigma : X(n) \rightarrow [0,1]$ such that $f_\sigma$ is supported on a ball around $\sigma$,  and 
$$T_\pi^k \phi (\sigma) = \sum_{\sigma ' \in X(n)} f^k_\sigma (\sigma ') \phi (\sigma ')$$
for every $\phi \in C (X(n), \pi)$.  Recall that for every $x \in \B$,  $R_1 x = \phi_x$ where 
$\phi_x (\sigma) = \pi ( \frac{ \mathbbm{1}_{G_{\sigma}}}{\mu (G_\sigma )}) x$. Also recall that  
$R_2 \phi = \frac{1}{\vert D(n) \vert} \sum_{\sigma \in D(n)} \phi (\sigma)$.
Thus for every $x \in \B$
\begin{dmath*}
R_2 T_\pi^k R_1 x = 
\frac{1}{\vert D(n) \vert} \sum_{\sigma \in D(n)} T_\pi^k \phi_x (\sigma) = 
\frac{1}{\vert D(n) \vert} \sum_{\sigma \in D(n)} \sum_{\sigma ' \in \supp ( f^k_\sigma)} f^k_\sigma (\sigma ') \phi_x (\sigma ') =
\frac{1}{\vert D(n) \vert} \sum_{\sigma \in D(n)}  \sum_{\sigma ' \in \supp ( f^k_\sigma)}  f^k_\sigma (\sigma ')  \pi ( \frac{ \mathbbm{1}_{G_{\sigma '}}}{\mu (G_\sigma ')}) x = 
\pi \left(\frac{1}{\vert D(n) \vert} \sum_{\sigma \in D(n)} \sum_{\sigma ' \in \supp ( f^k_\sigma)}  f^k_\sigma (\sigma ')  \frac{ \mathbbm{1}_{G_{\sigma '}}}{\mu (G_\sigma ')} \right) x
\end{dmath*}
and we take 
$$f_k = \frac{1}{\vert D(n) \vert} \sum_{\sigma \in D(n)} \sum_{\sigma ' \in \supp ( f^k_\sigma)}  f^k_\sigma (\sigma ')  \frac{ \mathbbm{1}_{G_{\sigma '}}}{\mu (G_\sigma ')}.$$
\end{proof}

\begin{theorem}[\.Zuk type criterion]
\label{Zuk type thm}
Let $G$ be a locally compact, unimodular group and $X$ be a pure $n$-dimensional, partite simplicial complex such that $X$ is gallery connected and all the $1$-dimensional links of $X$ are connected.  Assume that $G$ is acting on $X$ by simplicial automorphisms and the the action is cocompact and that for every $\tau \in X(n-2) \cup X(n-1) \cup X(n)$, $G_\tau$ is an open compact subgroup.   Let $\mathcal{E}_X$ be the class of Banach spaces such that for every $\B \in \mathcal{E}_X$ it holds that
$$\max_{\tau \in D (n-2)} \lambda_{\tau,   \bipartite}^\B  < \frac{1}{8n-3}.$$
Also, for $\varepsilon >0$, let 
$\mathcal{E}_{X, \varepsilon}$ be the class of Banach spaces such that for every $\B \in \mathcal{E}_X$ it holds that
$$\max_{\tau \in D (n-2)} \lambda_{\tau,   \bipartite}^\B  \leq \frac{1}{8n-3} - \varepsilon.$$
Then for every $\varepsilon >0$,  $G$ has robust property (T) with respect to $\mathcal{E}_{X, \varepsilon}$. Also, if $\mathbb{C} \in \mathcal{E}_X$, then $G$ property $(F \mathcal{E}_X)$. 
\end{theorem}

\begin{proof}
Note that by Lemma \ref{L2 norm stability},  $\mathcal{E}_X$ is closed under $\ell_2$ sums. Thus if $\mathbb{C} \in \mathcal{E}_X$ it follows from Proposition \ref{robust implies fp prop} that if $G$ has robust property (T) with respect to $\mathcal{E}_{X, \varepsilon}$, then $G$ has property $(F \mathcal{E}_{X, \varepsilon})$.  Therefore if we show that $G$ has robust property (T) with respect to $\mathcal{E}_{X, \varepsilon}$ for every $\varepsilon >0$, it will follow that if   $\mathbb{C} \in \mathcal{E}_X$ then $G$ has property $(F \mathcal{E}_X)$,  since $ \mathcal{E}_X = \bigcup_{\varepsilon >0} \mathcal{E}_{X, \varepsilon}$.

Fix $\varepsilon >0$.  Denote 
$$\beta_1 = \sqrt{\frac{ \frac{1}{8n-3} - \frac{\varepsilon}{2}}{ \frac{1}{8n-3} - \varepsilon}},  $$
$$\gamma =  \frac{1}{8n-3} - \frac{\varepsilon}{2},$$
and
$$\beta_2 = \sqrt{ 1+ \frac{1}{2} \frac{1-(8n-3)\gamma}{n-1 + (3n-1)\gamma}}.$$
Further denote $\beta = \min \lbrace \beta_1, \beta_2 \rbrace$ and note that $\beta >1$. 

Recall that $K \subseteq G$ is a compact symmetric generating set of $G$. We will show that there is a Kazhdan projection with respect to $\mathcal{F} (\mathcal{E}_{X, \varepsilon},K,\beta)$. Explicitly, let $f_k \in C_c (G)$ be the probability functions in Lemma \ref{f_k lemma}, i.e.,  the functions such that for every representation $\pi$, 
$$\pi (f_k) = R_2 T_\pi^k R_1,$$
where $T = \frac{1}{n+1} \sum_{\nu \subseteq \lbrace 0,...,n \rbrace,\vert \nu \vert  =n} P_\nu^\pi$. We will show that $(f_k)$ converges in $C_{\mathcal{F} (\mathcal{E}_{X, \varepsilon},K,\beta)}$ to $p$ and $\forall (\pi, \B) \in \mathcal{F} (\mathcal{E},K,\beta)$, $\pi (p)$ is a projection on $\B^{\pi (G)}$.

Fix $\pi \in \mathcal{F} (\mathcal{E}_{X, \varepsilon},K,\beta)$.  By Lemma \ref{projection norm bound lemma} and the choice of $\beta_2$, it holds for every $\nu \subseteq \lbrace 0,...,n \rbrace$, $\vert \nu \vert  =n$ that 
$$\Vert P_\nu^\pi \Vert \leq \beta_2^2 < 1+ \frac{1-(8n-3)\gamma}{n-1 + (3n-1)\gamma}.$$
By Lemma \ref{angle bound by bipartite link lemma} and the choice of $\beta_1$, it holds for every  $\nu, \nu ' \subseteq \lbrace 0,...,n \rbrace, \nu \neq \nu '$, $\vert \nu \vert = \vert \nu ' \vert =n$ that 
$$\cos (\angle (P_\nu^\pi, P_{\nu '}^\pi)) \leq \beta_1^2 \left( \frac{1}{8n-3} - \varepsilon \right) = \gamma < \frac{1}{8n-3}.$$
It follows that the set of projections $\lbrace P_\nu^\pi : \nu \subseteq \lbrace 0,...,n \rbrace,\vert \nu \vert  =n \rbrace$ fulfil the conditions of Theorem \ref{angle criterion thm}.  

Denote as above $T_\pi = \frac{1}{n+1} \sum_{\nu \subseteq \lbrace 0,...,n \rbrace,\vert \nu \vert  =n} P_\nu^\pi$. By Theorem \ref{angle criterion thm} there is a projection $T_\pi^\infty$ on $\bigcap_{\nu \subseteq \lbrace 0,...,n \rbrace,\vert \nu \vert  =n} \im (P_\nu^\pi)$ and constants $0\leq r (\gamma, \beta ) <1, C (\gamma, \beta ) >0$, such that for every $k$, 
$$\Vert T_\pi^k - T_\pi^\infty \Vert \leq Cr^k.$$
Thus, it holds for every $k$ that 
\begin{dmath*}
\Vert R_2 T_\pi^\infty R_1 - \pi (f_k) \Vert \leq \Vert R_2 \Vert \Vert R_1 \Vert C r^k \leq^{\text{Proposition } \ref{bound R_1, R_2 prop}} \\
\left( \beta \sqrt{\sum_{\sigma \in D(n)} \frac{1}{\mu (G_\sigma )}} \max_{\sigma \in D(n)} \sqrt{\mu (G_\sigma )} \right) C r^k.
\end{dmath*}
This shows that for every $\pi \in \mathcal{F} (\mathcal{E}_{X, \varepsilon},K,\beta)$, the sequence $\pi (f_k)$ converges in norm and the rate of convergence is bounded independently of $\pi$. Thus, the sequence $(f_k)$ converges in $C_{\mathcal{F} (\mathcal{E}_{X, \varepsilon},K,\beta)}$.  

It remains to show that for every $(\pi, \B) \in \mathcal{F} (\mathcal{E}_{X, \varepsilon},K,\beta)$, the operator $R_2 T_\pi^\infty R_1$ is a projection on $\B^{\pi (G)}$.  First, let $x \in \B^{\pi (G)}$. By the definition of $\phi_x$, $R_1 x = \phi_x$ is the constant function $\phi_x \equiv x$. By Lemma \ref{intersection of images is const lemma}, $T_\pi^\infty$ is a projection on the subspace of constant functions in $C(X(n), \pi)$ and thus $T_\pi^\infty \phi_x = \phi_x$. By the definition of $R_2$, $R_2 \phi_x = x$ (since $\phi_x \equiv x$). Combining all of the above, it follows that for every $x \in \B^{\pi (G)}$, $R_2 T^\infty R_1 x = x$. 

Second, for every $x \in \B$, $T_\pi^\infty R_1 x = T_\pi^\infty \phi_x$ and by Lemma \ref{intersection of images is const lemma} it follows that $T_\pi^\infty \phi_x$ is a constant equivariant function in $C(X(n), \pi)$, i.e., there is $x_0 \in \B^{\pi (G)}$ such that $T_\pi^\infty \phi_x \equiv x_0$. Thus, $R_2 T_\pi^\infty \phi_x = x_0 \in \B^{\pi (G)}$ and it follows that $\im (R_2 T_\pi^\infty R_1) \subseteq \B^{\pi (G)}$. Therefore, we can conclude that $R_2 T^\infty R_1$ is a projection on $\B^{\pi (G)}$ as needed.
\end{proof}

\section{Application to random groups in the Gromov density model}

\label{random groups sec}

The aim of the Section is to apply of Banach \.Zuk criterion in the setting of random groups in the Gromov density model and prove an extended version of Theorem \ref{f.p. density model - intro} that appeared in the introduction.  Some definitions below already appeared in the introduction and we recall them for completeness.

A random group is a group chosen randomly according to some model and one is interested in the asymptotic properties of such randomly chosen group.  The most famous model is the Gromov density model: 
\begin{definition}[Gromov density model]
Let $k \in \mathbb{N}, k \geq 2$ and $0 \leq d \leq 1$ be constants and $l \in \mathbb{N}$ be a parameter.  A random group is the Gromov density model $\mathcal{D} (k,l,d)$ is a group $\Gamma = \langle \mathcal{A}  \vert \mathcal{R} \rangle$ where $\vert \mathcal{A}  \vert =k$ and $\mathcal{R}$ is a set of relators of length $l$ (in $\mathcal{A} \cup \mathcal{A} ^{-1}$) randomly chosen from the set 
$$\lbrace \mathcal{R} \text{ is a set of cyclically reduced relators of length } l :  \vert \mathcal{R} \vert = \lfloor (2k-1)^{dl} \rfloor \rbrace$$ with uniform probability. We denote a random group in this model by $\Gamma \in \mathcal{D} (k,l,d)$.

For a group property $P$, we say that $P$ holds asymptotically almost surely (a.a.s.) in $\mathcal{D} (k,l,d)$ if 
$$\lim_{l \rightarrow \infty} \mathbb{P} (\Gamma \in \mathcal{D} (k,l,d) \text{ has property } P) = 1.$$
\end{definition}

We will show that for certain class of Banach spaces $\mathcal{E}$ property $(F \mathcal{E})$ holds a.a.s. in $\mathcal{D} (k,l,d)$ (see exact formulation below).  

A closely related model to the density model is the Binomial model:
\begin{definition}[Gromov binomial model]
Let $k \in \mathbb{N}, k \geq 2$ be a constant, $\rho : \mathbb{N} \rightarrow [0,1]$ be a function and $l$ be a parameter.   A random group is the Gromov binomial model $\mathcal{B} (k,l,\rho)$ is a group $\Gamma = \langle \mathcal{A} \vert \mathcal{R} \rangle$ where $\vert \mathcal{A}  \vert =k$ and $\mathcal{R}$ is a set of relators of cyclically reduced  length $l$ (in $\mathcal{A}  \cup \mathcal{A} ^{-1}$) where each relator in $\mathcal{R}$ chosen independently with probability $\rho (l)$. We denote a random group in this model by $\Gamma \in \mathcal{B} (k,l,\rho)$.

For a group property $P$, we say that $P$ holds asymptotically almost surely (a.a.s.) in $\mathcal{B} (k,l,\rho)$ if 
$$\lim_{l \rightarrow \infty} \mathbb{P} (\Gamma \in \mathcal{B} (k,l,\rho) \text{ has property } P) = 1.$$
\end{definition}

For a group property $P$ we say that $P$ is \textit{monotone increasing} if it is preserved under quotients.  For our purposes it will be important that Proposition \ref{f.p preserve under quotient prop} implies that property $(F \mathcal{E})$ is monotone increasing. The following Proposition appears in \cite{JLR} (as part of a systematic study of random sets in both models): 
\begin{proposition}\cite[Proof of Proposition 1.13]{JLR}
Let $P$ be a monotone increasing group property and $k \geq 2, 0 < d <1$ be constants.  For $C \geq 1$,  we denote 
$$\rho_C (l) = (2k-1)^{-(1-d)l} - C  (2k-1)^{-(\frac{3}{2}-\frac{d}{2})l} .$$
Then
$$\mathbb{P} (\Gamma \in \mathcal{D} (k,l,d) \text{ has property } P) \geq \mathbb{P} (\Gamma \in \mathcal{B} (k,l,\rho_C) \text{ has property } P) -  \frac{3}{C}.$$
\end{proposition}

The following Corollary readily follows:
\begin{corollary}
\label{from density to binomial coro}
Let $P$ be a monotone increasing group property and $k \geq 2, 0 < d <1$ be constants.   Denote 
$$\rho (l) = \frac{1}{2} (2k-1)^{-(1-d)l} .$$
Then
$$\mathbb{P} (\Gamma \in \mathcal{D} (k,l,d) \text{ has property } P) \geq \mathbb{P} (\Gamma \in \mathcal{B} (k,l,\rho) \text{ has property } P) -  \frac{6}{(2k-1)^{\frac{l}{2}}}.$$
\end{corollary}

Following the ideas of \cite{KK} and assuming that $l$ is divisible by $3$, we introduce the model $\mathcal{B} ' (k,l,\rho)$ as a variant of  $\mathcal{B} (k,l,\rho)$ defined as follows: Let $W_{\frac{l}{3}} '$ be reduced words in $\mathcal{A} \cup \mathcal{A}^{-1}$ of length $\frac{l}{3}$ with a first and last letter in $\mathcal{A}$.  Define $W'_l$ to be words in $\mathcal{A} \cup \mathcal{A}^{-1}$ of length $l$ that are concatenations of $3$ words in $W_{\frac{l}{3}} '$.  
\begin{definition}
Let $k \geq 2$ be a constant and $\rho : \mathbb{N} \rightarrow [0,1]$ be a function.  A random group is the model $\mathcal{B} ' (k,l,\rho)$ is a group $\Gamma = \langle \mathcal{A} \vert \mathcal{R} \rangle$ where $\vert \mathcal{A}  \vert =k$ and $\mathcal{R}$ is a set of relators chosen from $W_l'$ independently with probability $\rho (l)$.
\end{definition}

\begin{proposition}
\label{B' prop}
Let $P$ be a monotone increasing group property, $k \geq 2$ be a constant and $\rho : \mathbb{N} \rightarrow [0,1]$ be a function. If a random group in $\mathcal{B} ' (k,l,\rho)$  has $P$ a.a.s., then a random group in $\mathcal{B} (k,l,\rho)$ has $P$ a.a.s.  More explicitly, 
$$\mathbb{P} (\Gamma \in \mathcal{B} (k,l,\rho) \text{ has property } P)  \geq \mathbb{P} (\Gamma ' \in \mathcal{B}' (k,l,\rho) \text{ has property } P) .$$
\end{proposition}

\begin{proof}
Denote $W_l$ to be the set of cyclically reduced words of length $l$ in $\mathcal{A} \cup \mathcal{A}^{-1}$. The relators of $\Gamma \in \mathcal{B} (k,l,\rho)$ are chosen independently from $W_l$ and thus one can start by first choosing relators from $W_l'$ and then from $W_l \setminus W_l'$.  Thus it is clear that $\Gamma \in \mathcal{B} (k,l,\rho)$ is a quotient of $\Gamma ' \in \mathcal{B} ' (k,l,\rho)$ and by the assumption that $P$ is monotone increasing
$$\mathbb{P} (\Gamma \in \mathcal{B} (k,l,\rho) \text{ has } P) \geq \mathbb{P} (\Gamma ' \in \mathcal{B} ' (k,l,\rho) \text{ has } P).$$
\end{proof}

The above Proposition shows that in order to prove property $(F \mathcal{E})$ holds a.a.s.  for $\mathcal{B} (k,l,\rho)$ (and thus for $\mathcal{D} (k,l,d)$ for $\rho (l) =  \frac{1}{2}(2k-1)^{-(1-d)l}$ by Corollary \ref{from density to binomial coro}) it is sufficient to prove that property $(F \mathcal{E})$ holds a.a.s.  for $\mathcal{B} ' (k,l,\rho)$. 

Last, again following \cite{KK}, we define the model $\mathcal{M}^+ (m, \rho)$ as follows: 
\begin{definition}[The positive triangular model]
Let $m \in \mathbb{N}$ be a parameter and $\rho : \mathbb{N} \rightarrow [0,1]$ be a function.  A random group is the model $\mathcal{M}^+ (m,\rho)$ is a group $\Gamma = \langle S  \vert R \rangle$ where $\vert S  \vert =m, S \cap S^{-1} = \emptyset$ and $R$ is a set of length $3$ relators in the letters of $S$ (not using the letters of $S^{-1}$) chosen independently with probability $\rho (m)$. 
\end{definition}

Fix some $k \geq 2$ and $\mathcal{A}$ with $\vert \mathcal{A} \vert = k$.  Let $W_{\frac{l}{3}}', W_l'$ be as above (with respect to $\mathcal{A}$).  For $m = \vert W_{\frac{l}{3}}' \vert$ and $S = \lbrace s_1,...,s_m \rbrace$ fix a bijection $\varphi : S \rightarrow W_{\frac{l}{3}}'$.  This bijection extends to a bijection $\varphi : \lbrace s_{i_1} s_{i_2} s_{i_3} : 1 \leq i_1, i_2, i_3 \leq m \rbrace \rightarrow W_l'$.  Thus, for $\Gamma \in \mathcal{M}^+ (m, \rho), \Gamma = \langle S \vert R \rangle$  the map $\varphi :  \Gamma \rightarrow \langle \mathcal{A} \vert \varphi ( R) \rangle$ extends to a homomorphism. As the following Lemma states, $\varphi (\Gamma)$ is a finite index subgroup of $ \langle \mathcal{A} \vert \varphi ( R) \rangle$.

\begin{lemma}\cite[Lemma 3.15]{KK}
\label{KK f.i.  lemma}
For every $\Gamma \in \mathcal{M}^+ (m, \rho)$, $\varphi (\Gamma)$ is a finite index subgroup of $\langle \mathcal{A} \vert \varphi ( R) \rangle$.
\end{lemma}

\begin{remark}
The random group models considered in \cite{KK} are the density models, but the proof of \cite[Lemma 3.15]{KK} passes verbatim to the binomial models.
\end{remark}

All the discussion above reduces the problem of proving $(F \mathcal{E})$ to $\mathcal{D} (k, l, d)$ to proving $(F \mathcal{E})$ to $\mathcal{M}^+ (m, \rho)$ (with the correct choice of $\rho$): 
\begin{theorem}
\label{reduction thm}
Let $\frac{1}{3} <d < 1$, $k \geq 2$ be constants and $\mathcal{E}_l$ be a class of Banach spaces (this class may change as $l$ increases).  Define $m (l) = \frac{1}{2} (2k-1)^{\frac{l}{3}}$ and $\rho : \mathbb{N} \rightarrow [0,1]$ to be $\rho (m) = \frac{1}{4 m^{3(1-d)}}$.  For every $l$ divisible by $3$, 
\begin{dmath*}
\mathbb{P} (\Gamma \in \mathcal{D} (k,l,d) \text{ has property } (F \mathcal{E}_l)) \geq  
{\mathbb{P} (\Gamma \in \mathcal{M}^+ (m (l) , \rho) \text{ has property } (F \mathcal{E}_l)) -  \frac{6}{(2k-1)^{\frac{l}{2}}}.}
\end{dmath*}
\end{theorem}

\begin{proof}
By Corollary \ref{from density to binomial coro},  for $\rho_1 (l) = \frac{1}{2}(2k-1)^{-(1-d)l}$ it hold that
$$\mathbb{P} (\Gamma \in \mathcal{D} (k,l,d) \text{ has property } (F \mathcal{E}_l)) \geq \mathbb{P} (\Gamma \in \mathcal{B} (k,l,\rho_1) \text{ has property } (F \mathcal{E}_l)) -  \frac{6}{(2k-1)^{\frac{l}{2}}}.$$
By Proposition \ref{B' prop},  
$$\mathbb{P} (\Gamma \in \mathcal{B} (k,l,\rho_1) \text{ has property } (F \mathcal{E}_l))  \geq \mathbb{P} (\Gamma ' \in \mathcal{B}' (k,l,\rho_1) \text{ has property } (F \mathcal{E}_l)) .$$
Thus we are left to show that for every $l$ divisible by $3$ is holds that
$$\mathbb{P} (\Gamma ' \in \mathcal{B}' (k,l,\rho_1) \text{ has property } P) \geq {\mathbb{P} (\Gamma \in \mathcal{M}^+ (m (l) , \rho) \text{ has property } (F \mathcal{E}_l)),}$$
with 
$m (l) = \frac{1}{2} (2k-1)^{\frac{l}{3}}$ and $\rho (m) =  \frac{1}{4 m^{3(1-d)}}$. 

Assume that $l$ is divisible by $3$. By Lemma \ref{KK f.i. lemma}, for every $\Gamma ' \in \mathcal{B} ' (k,l,\frac{1}{2} (2k-1)^{-(1-d)l})$ there is a group $\Gamma \in  \mathcal{M}^+ (m_1 , \rho_2)$ with $m_1 (l)  = \vert W_{\frac{l}{3}}' \vert = k^2 (2k-1)^{\frac{l}{3}-2}$ and $\rho_2 (m _1) = \frac{1}{2} (2k-1)^{-(1-d)l}$ such that $\varphi (\Gamma)$ is a finite index subgroup in $\Gamma '$ and $\Gamma$, $\Gamma '$ occur in the same probability.  

By Proposition \ref{f.p. is equivalent to f.i. prop},  if $\Gamma$ has property $(F \mathcal{E}_l)$, then $\Gamma '$ has property  $(F \mathcal{E}_l)$.  Using the fact that property $(F \mathcal{E}_l)$ is monotone increasing it follows that we can replace $m_1$ by 
$$m = \frac{1}{2} (2k-1)^{\frac{l}{3}} \geq m_1,$$
i.e., 
\begin{dmath*}
\mathbb{P} (\Gamma ' \in \mathcal{B}' (k,l,\rho_1) \text{ has property } P) \geq 
{\mathbb{P} (\Gamma \in \mathcal{M}^+ (m_1  , \rho_2 (m_1)) \text{ has property } (F \mathcal{E}_l))}
 \geq {\mathbb{P} (\Gamma \in \mathcal{M}^+ (m  , \rho_2 (m_1)) \text{ has property } (F \mathcal{E}_l)).}
 \end{dmath*}

Note that
\begin{dmath*}
\rho_2 (m_1) = \frac{1}{2} (2k-1)^{-(1-d)l} =\frac{1}{2}  \frac{1}{((2k-1)^{\frac{l}{3}})^{3(1-d)}} = \frac{1}{2} \frac{1}{ 2 m^{3(1-d)}} = 
\frac{1}{4 m^{3(1-d)}}.
\end{dmath*}
It follows that for $\rho (m) = \frac{1}{4 m^{3(1-d)}}$  it holds that
$$\mathbb{P} (\Gamma ' \in \mathcal{B}' (k,l,\rho_1) \text{ has property } P) \geq {\mathbb{P} (\Gamma \in \mathcal{M}^+ (m  , \rho (m)) \text{ has property } (F \mathcal{E}_l)),}$$
as needed.
\end{proof}

Theorem \ref{reduction thm} reduces the proof of Banach fixed point properties in the Gromov density to proving Banach fixed properties in the positive triangular model. This model is convenient for us, since its Cayley complex is a two-dimensional simplicial complex with bipartite links and thus we can apply our Theorem \ref{Zuk type thm} above. For this, we will need to bound the second eigenvalue of the random walk on the link of the Cayley complex of a group in $\mathcal{M}^+ (m, \rho)$. Our approach for bounding the second eigenvalue is heavily based on \cite{ALS} and \cite{LaatSalle} and we claim very little originality here (a few adaptations were needed in order to provide a sharp result in the triangular positive model). In order to bound the second eigenvalue, we will first need some general lemmata regarding graphs and the bipartite Erd\"{o}s-R\'{e}nyi graph.  


\begin{lemma}
\label{several graphs lemma}
Let $(V,E_1),...,(V,E_k)$ be connected graphs on the same vertex set $V$. Assume that $E_1,...,E_k$ are pairwise disjoint and that there are constants $d, 0 < \varepsilon <1$ such that for every $i$, the degree of a vertex in $(V,E_i)$ is between $d(1- \varepsilon)$ and $d(1+\varepsilon)$. Assume there is $\lambda <1$ such that the second eigenvalue of the random walk on each $(V, E_i)$ is less than $\lambda$. Then the second eigenvalue of the random walk on $(V, E_1 \cup ... \cup E_k)$ is less that $\lambda +\frac{16 \varepsilon^2}{(1- \varepsilon)^4}$. 
\end{lemma}

\begin{proof}
For $v \in V$, denote $m_i (v)$ to be the degree of $v$ in $(V, E_i)$ and $m(v)$ to be the degree of $v$ in $(V, E_1 \cup ... E_k)$. We need to prove that for $\phi : V \rightarrow \mathbb{R}$, if $\sum_{v \in V} m (v) \phi (v) = 0$, then 
$$\sum_{v \in V} m(v) (M \phi (v)) \phi (v) \leq \left(\lambda + \frac{16 \varepsilon^2}{(1- \varepsilon)^4} \right) \sum_{v \in V} m (v) \vert \phi (v) \vert^2,$$
where $M$ is the random walk operator on $(V, E_1 \cup ...  \cup E_k)$.  Equivalently, we need to show that for the normalized Laplacian $ I-M$ it holds for $\phi$ as above that 
$$\sum_{v \in V} m(v) ((I-M) \phi (v)) \phi (v) \geq \left(1-\lambda - \frac{16 \varepsilon^2}{(1- \varepsilon)^4} \right) \sum_{v \in V} m (v) \vert \phi (v) \vert^2.$$
Recall that 
$$\sum_{v \in V} m(v) ((I-M) \phi (v)) \phi (v) = 
\sum_{\lbrace u,v \rbrace \in E_1 \cup ...  \cup E_k} \vert \phi (u) - \phi (v) \vert^2$$
where double edges (if they exist) are added several times according to their multiplicity. 
Similarly, if we denote $M_i$ to be the random walk operator on $(V,E_i)$ then for every $i$
$$\sum_{v \in V} m(v) ((I-M_i) \phi (v)) \phi (v) = 
\sum_{\lbrace u,v \rbrace \in E_i} \vert \phi (u) - \phi (v) \vert^2.$$ 

Thus
\begin{dmath*}
\sum_{v \in V} m(v) ((I-M) \phi (v)) \phi (v) = 
\sum_{\lbrace u,v \rbrace \in E_1 \cup ...  \cup E_k} \vert \phi (u) - \phi (v) \vert^2 = 
\sum_{i=1}^k \sum_{\lbrace u,v \rbrace \in E_i} \vert \phi (u) - \phi (v) \vert^2 =
\sum_{i=1}^k  \sum_{v \in V} m(v) ((I-M_i) \phi (v)) \phi (v) \geq 
(1- \lambda) \sum_{i=1}^k \sum_{v \in V} m_i (v)\left\vert \phi (v) - \frac{1}{m_i (V)} \sum_{u \in V} m_i (u) \phi (u) \right\vert^2 =
(1- \lambda) \sum_{i=1}^k \sum_{v \in V} m_i (v) \vert \phi (v) \vert^2 - (1- \lambda) \sum_{i=1}^k \sum_{v \in V} m_i (v) \left\vert \sum_{u \in V} \frac{m_i (u)}{m_i (V)} \phi (u) \right\vert^2 = 
.(1- \lambda) \sum_{v \in V} m (v) \vert \phi (v) \vert^2 - (1- \lambda)  \sum_{i=1}^k m_i (V) \left\vert \sum_{u \in V} \frac{m_i (u)}{m_i (V)} \phi (u) \right\vert^2.
\end{dmath*}

Thus it is sufficient to prove that 
$$ \sum_{i=1}^k m_i (V) \left\vert \sum_{u \in V} \frac{m_i (u)}{m_i (V)} \phi (u) \right\vert^2 \leq \frac{16 \varepsilon^2}{(1- \varepsilon)^4} \sum_{v \in V} m (v) \vert \phi (v) \vert^2 .$$
Fix $1 \leq i \leq k$. Note that for every $u \in V$, 
$$\frac{(1-\varepsilon)d}{k (1+\varepsilon)d} \leq \frac{m_i (u)}{m(u)} \leq \frac{(1+\varepsilon)d}{k (1-\varepsilon)d}.$$
Thus
$$ \left( \frac{1-\varepsilon}{1+\varepsilon} \right)^2 \leq \frac{m(u)}{m_i (u)} \frac{m_i (V)}{m(V)} \leq \left( \frac{1+\varepsilon}{1-\varepsilon} \right)^2.$$
Therefore 
$$\left\vert 1 - \frac{m(u)}{m_i (u)} \frac{m_i (V)}{m(V)}  \right\vert \leq \max \left\lbrace  \left( \frac{1+\varepsilon}{1-\varepsilon} \right)^2 - 1, 1- \left( \frac{1-\varepsilon}{1+\varepsilon} \right)^2 \right\rbrace = \frac{4 \varepsilon}{(1- \varepsilon)^2}.$$
It follows that 
\begin{dmath*}
\sum_{i=1}^k m_i (V) \left\vert \sum_{u \in V} \frac{m_i (u)}{m_i (V)} \phi (u) \right\vert^2 = 
\sum_{i=1}^k m_i (V) \left\vert \sum_{u \in V} \left( \frac{m_i (u)}{m_i (V)} - \frac{m(u)}{m (V)} \right) \phi (u) \right\vert^2 = 
{\sum_{i=1}^k m_i (V) \left\vert \sum_{u \in V} \frac{m_i (u)}{m_i (V)} \left( 1 - \frac{m(u)}{m_i (u)} \frac{m_i (V)}{m(V)} \right) \phi (u) \right\vert^2 }\leq
{\sum_{i=1}^k m_i (V) \sum_{u \in V} \frac{m_i (u)}{m_i (V)}   \left\vert \left( 1 - \frac{m(u)}{m_i (u)} \frac{m_i (V)}{m(V)} \right) \phi (u) \right\vert^2 } \leq  
{\sum_{i=1}^k \sum_{u \in V} m_i (u)    \left(  \frac{4 \varepsilon}{(1- \varepsilon)^2}  \right)^2  \vert \phi (u) \vert^2 }=
 \frac{16  \varepsilon^2}{(1- \varepsilon)^4} \sum_{u \in V} m (u)   \vert \phi (u) \vert^2,
\end{dmath*}
as needed.

\end{proof}

\begin{lemma}
\label{graph deleting lemma}
Let $(V, E_1), (V,E_2)$ be finite graphs on the same vertex set $V$ and $0 \leq \lambda <1$, $ 0 \leq \varepsilon <1$ be a constants. Denote $m_i (v)$ to be the degree of $v$ in $(V,E_i)$. Assume that $(V,E_1 \cup E_2)$ is connected and that the second eigenvalue of the random walk on it is less than $\lambda$.  If for every $v \in V$, $\frac{m_2 (v)}{m_1 (v)} \leq \varepsilon <\frac{1-\lambda}{4}$, then $(V, E_1)$ is connected and the second eigenvalue of the random walk on it is less than $\lambda + 4 \varepsilon$.
\end{lemma}

The proof is very similar to the proof of Lemma \ref{several graphs lemma} and thus we allow ourselves to omit some details.
\begin{proof}
Denote $M_i, i=1,2$ to be the random walk operator on $(V,E_i)$.  Let $\phi : V \rightarrow \mathbb{R}$ such that $\sum_{v} m_1 (v) \phi (v) =0$.  We need to show that 
$$\sum_{v \in V} m_1 (v) (M_1 \phi  (v) ) \phi (v) \leq \left( \lambda + 4 \varepsilon \right) \sum_{v \in V} m_1 (v) \vert \phi (v) \vert^2$$
or equivalently that 
$$\sum_{v \in V} m_1 (v) ((I-M_1) \phi  (v) ) \phi (v) \geq \left( 1-\lambda - 4 \varepsilon \right) \sum_{v \in V} m_1 (v) \vert \phi (v) \vert^2.$$

Observe that 
\begin{dmath*}
\sum_{v \in V} m(v) ((I-M) \phi (v)) \phi (v) = 
\sum_{v \in V} m_1 (v) ((I-M_1) \phi (v)) \phi (v) + \sum_{v \in V} m_2 (v) ((I-M_2) \phi (v)) \phi (v) \leq 
\sum_{v \in V} m_1 (v) ((I-M_1) \phi (v)) \phi (v) + 2 \sum_{v \in V} m_2 (v) \vert \phi (v) \vert^2 \leq 
\sum_{v \in V} m_1 (v) ((I-M_1) \phi (v)) \phi (v) + 2 \varepsilon \sum_{v \in V} m_1 (v) \vert \phi (v) \vert^2.
\end{dmath*}
Thus, 
$$\sum_{v \in V} m_1 (v) ((I-M_1) \phi (v)) \phi (v)  \geq \sum_{v \in V} m(v) ((I-M) \phi (v)) \phi (v) - 2 \varepsilon \sum_{v \in V} m_1 (v) \vert \phi (v) \vert^2$$
and we are left to prove that
$$\sum_{v \in V} m(v) ((I-M) \phi (v)) \phi (v)  \geq (1- \lambda - 2\varepsilon) \sum_{v \in V} m_1 (v) \vert \phi (v) \vert^2.$$

Note that for every $u \in V$,  $\frac{1}{1+ \varepsilon} \leq \frac{m_1 (u)}{m (u)} \leq 1$ and
$$\frac{1}{1+ \varepsilon} \leq \frac{m (V)}{m(u)} \frac{m_1 (u)}{m_1 (V)} \leq 1+ \varepsilon.$$
Thus
\begin{dmath*}
\sum_{v \in V} m(v) ((I-M) \phi (v)) \phi (v) \geq 
(1- \lambda) \sum_{v \in V} m(v) \vert \phi (v) \vert^2 - (1- \lambda)  m (V) \left\vert \sum_{u \in V} \frac{m (u)}{m(V)} \phi (u) \right\vert^2 \geq (1- \lambda) \sum_{v \in V} m_1 (v) \vert \phi (v) \vert^2 -  m (V) \left\vert \sum_{u \in V} \frac{m (u)}{m(V)} \phi (u) \right\vert^2 =
(1- \lambda) \sum_{v \in V} m_1 (v) \vert \phi (v) \vert^2 -  m (V) \left\vert \sum_{u \in V} \frac{m (u)}{m(V)} \left(1 -   \frac{m (V)}{m(u)} \frac{m_1 (u)}{m_1 (V)} \right) \phi (u) \right\vert^2 \leq 
(1- \lambda) \sum_{v \in V} m_1 (v) \vert \phi (v) \vert^2 -  \varepsilon^2 \sum_{u \in V} m (u) \vert  \phi (u) \vert^2  \leq 
(1- \lambda) \sum_{v \in V} m_1 (v) \vert \phi (v) \vert^2 -  \varepsilon^2 (1+\varepsilon) \sum_{u \in V} m_1 (u) \vert  \phi (u) \vert^2 \leq 
(1- \lambda) \sum_{v \in V} m_1 (v) \vert \phi (v) \vert^2 -  2\varepsilon \sum_{u \in V} m_1 (u) \vert  \phi (u) \vert^2,
\end{dmath*}
as needed.
\end{proof}

For $n \in \mathbb{N}$ and $\rho : \mathbb{N} \rightarrow [0,1]$, let $\mathbb{G}_\bipartite (n, \rho)$ be the bipartite Erd\"{o}s-R\'{e}nyi graph defined as follows: The vertex set of $V$ is $S_1 \cup S_2$ where $S_1, S_2$ are disjoint sets of size $n$.  We randomly choose edges in the set $\lbrace \lbrace u,v \rbrace : u \in S_1, v \in S_2 \rbrace$ where each edge is chosen independently with probability $\rho (n)$.  

\begin{theorem}\cite[Theorem A]{Ash}
\label{second ev of bipart ER thm}
If $\rho (n) \geq  \frac{\log^6 (n)}{n}$, then with probability tending to $1$ as $n$ grows to infinity the graph $\mathbb{G}_\bipartite (n, \rho)$ is connected and its second largest eigenvalue is less than $\frac{8}{\sqrt{n \rho (n)}}$. 
\end{theorem}

\begin{lemma}
\label{degree dist lemma}
Let $0 \leq \alpha <1$ be a constant. Assume $\rho (n) \geq  \frac{C}{n^\alpha}$ for some positive constant $C$. Then
$$\lim_n \mathbb{P} (\vert m(v) - n \rho (n) \vert \geq \frac{n \rho (n)}{n^{\frac{1-\alpha}{3}}} \text{ for some } v \text{ in } \mathbb{G}_\bipartite (n, \rho) ) = 0.$$
\end{lemma}

\begin{proof}
For every vertex $v$, $m(v)$ is a binomial random variable.  For every $0< \delta < 1$ it holds by Chernoff bound that
$$\mathbb{P} (m(v) \geq (1+ \delta) n \rho (n)) \leq \exp (-\frac{\delta^2}{3} n \rho (n)),$$
$$\mathbb{P} (m(v) \leq (1- \delta) n \rho) \leq \exp (-\frac{\delta^2}{3} n \rho (n)).$$   

Thus if we choose $\delta = \frac{1}{n^{\frac{1-\alpha}{3}}}$, we have that for every large $n$
$$\mathbb{P} \left(\vert m(v) - n \rho (n) \vert \geq \frac{n \rho (n)}{n^{\frac{1-\alpha}{3}}} \right) \leq 2 \exp \left(-\frac{n \rho (n)}{3 n^{\frac{2(1-\alpha)}{3}}} \right) \leq  2 \exp \left( -\frac{C n^{\frac{1-\alpha}{3}}}{3} \right).$$
By union bound, the probability the some vertex $v$ has $\vert m(v) - n \rho (n) \vert \geq \frac{n \rho (n)}{n^{\frac{1-\alpha}{3}}} $ is less than
$4n \exp \left( -\frac{C n^{\frac{1-\alpha}{3}}}{3} \right)$ and this tends to $0$ as $n$ tends to infinity.
\end{proof}

After this set-up, we argue as in \cite{ALS}: Let $\Gamma \in \mathcal{M}^+ (m, \rho)$.  The relations of $\Gamma$ are of the form $s_i s_j s_k$ and thus the Cayley complex of $\Gamma$ is a two-dimensional simplicial complex on which $\Gamma$ acts transitively on the vertices.  The link of $e \in \Gamma$ is a bipartite graph with sides $S_0 = S, S_1 = S^{-1}$, where each relation $s_i s_j s_k$ gives rise to $3$ edges $\lbrace s_j, s_i^{-1} \rbrace,  \lbrace s_k, s_j^{-1} \rbrace, \lbrace s_i, s_k^{-1} \rbrace$. Thus the the link is composed of $3$ bipartite graphs $L_1, L_2, L_3$ on $S \cup S^{-1}$: 
\begin{enumerate}
\item The graph $L_1$ is composed of all the edges $\lbrace s_j, s_i^{-1} \rbrace$ where  $s_i s_j s_k$ is a relation.
\item The graph $L_2$ is composed of all the edges  $\lbrace s_k, s_j^{-1} \rbrace$ where  $s_i s_j s_k$ is a relation.
\item The graph $L_3$ is composed of all the edges  $\lbrace s_i, s_k^{-1} \rbrace$ where  $s_i s_j s_k$ is a relation.
\end{enumerate}
Note that we allow double edges, i.e., if $s_i s_j s_k, s_i s_j s_{k'}$ are two relations of $\Gamma$, then the edge $\lbrace s_j, s_i^{-1} \rbrace$ will appear twice.

The idea of \cite{ALS} is that each of these graphs behave like $\mathbb{G}_\bipartite (m, \rho ')$ where $\rho ' \geq \rho$ and thus by Lemma \ref{several graphs lemma} and Theorem \ref{second ev of bipart ER thm}, we can bound the second eigenvalue of the link. However, some analysis is needed to account for double edges and this will require Lemma \ref{degree dist lemma}.  To sum up, we will prove the following:
\begin{theorem}
\label{link of Cayley complex thm}
Let $\frac{3}{2} < \alpha < 2$, $C> 0$ be constants and $\rho (m) =\frac{C}{m^{\alpha}}$.  For a random group $\Gamma \in \mathcal{M}^+ (m, \rho)$, it holds a.a.s. that link of each vertex in the Cayley complex of $\Gamma$ is a connected graph and the second eigenvalue of the random walk on it is $< \frac{10}{\sqrt{C} m^{1- \frac{\alpha}{2}}}$.
\end{theorem}

\begin{proof}
For clarity, in this proof we will denote the degree of a vertex by $w (v)$ and not $m (v)$, since the letter $m$ already appears as a parameter of the random group.

We first claim that a.a.s., for every $i=1,2,3$ each edge in $L_i$ appears with multiplicity $\leq 3$. Indeed, the probability that a specific edge appear $4$ times or more is bounded by $m^4 \rho (m)^4$. By union bound, the probability that some edge appears $4$ times or more is bounded by $m^6   \rho (m)^6 \leq C^4 m^{6-4 \alpha}$ and this tends to $0$ since $\alpha > \frac{3}{2}$.

Second,  we claim that a.a.s, for every for every $i=1,2,3$ each vertex in $L_i$ has at most $K = \lfloor \frac{1}{2 \alpha - 3} \rfloor + 1$  double/triple edges connected to it.  Indeed, the probability that a specific vertex has $K$ double/triple edges connected to it is bounded by $m^{3K} \rho (m)^{2K}$. By union bound, the probability that some vertex has $K$ double/triple edges connected to it is bounded by $2m^{3K+1} \rho (m)^{2K} \leq 2 C^{2K} m^{3K+1-2K \alpha}$ and this tends to $0$ since $K > \frac{1}{2 \alpha - 3}$.

Let $L_i '$ be the graph obtained from $L_i$ by deleting multiple edges. Then a.a.s., for every $v \in S \cup S^{-1}$, the degree of $v$ in $L_i '$ differ than the degree of $v$ in $L_i$ by at most $2 K$ (which is constant and does not depend on $m$).  

Each of the $L_i '$ is a random bipartite Erd\"{o}s-R\'{e}nyi graph $\mathbb{G}_\bipartite (m, \rho ')$ with 
$$\rho ' (m) = 1- (1-\rho (m))^m \geq \frac{C}{m^{\alpha -1}}$$
for every large $m$.

By Theorem \ref{second ev of bipart ER thm}, it holds a.a.s. for every $i =1,2,3$ that $L_i '$ is connected and the second eigenvalue of the random walk operator on it is bounded from above by $\frac{8}{ \sqrt{m \rho ' (m)}} \leq \frac{8}{\sqrt{C} m^{1- \frac{\alpha}{2}}}$.  By Lemma \ref{degree dist lemma}, for every $i$ and every vertex in $L_i '$,  it holds a.a.s. that
$$\vert w_i '(v) - m \rho ' (m) \vert \leq \frac{m \rho ' (m)}{m^{\frac{2-\alpha}{3}}},$$
where $w_i '$ denotes the degree of $v$ in $L_i '$. Thus, by Lemma \ref{several graphs lemma}, it holds a.a.s. that for every large $m$ the second largest eigenvalue of the random walk on $L_1 ' \cup L_2 ' \cup L_3 '$ is bounded from above by 
$$ \frac{8}{\sqrt{C} m^{1- \frac{\alpha}{2}}} + \frac{20}{m^{\frac{4}{3}(1-\frac{\alpha}{2})}} <  \frac{9}{\sqrt{C} m^{1- \frac{\alpha}{2}}} .$$

Note that by the discussion above $L_1 \cup L_2 \cup L_3$ is a union of $L_1 ' \cup L_2 ' \cup L_3 '$ and the graph $H$ of the deleted edges.  Also note that a.a.s.  it holds that every vertex $v$, $H$ has at most $6 K = 6 \lfloor \frac{1}{2 \alpha - 3} \rfloor + 6$ edges connected to it. Therefore if we denote $w'$ to be the degree of a vertex is $ L_1  ' \cup L_2 ' \cup L_3 '$ and $w''$ to be the degree of a vertex in $H$, it holds a.a.s.  for every large $m$ that 
$$\frac{w'' (v)}{w' (v)} \leq \frac{6K}{3 ( m \rho ' (m) - \frac{m \rho ' (m)}{m^{\frac{2-\alpha}{3}}})} \leq \frac{4K}{C m^{2 - \alpha}}.$$
Thus, from Lemma \ref{graph deleting lemma} it follows that a.a.s. the second largest eigenvalue of the random walk on $L_1 \cup L_2 \cup L_3 $ is bounded from above by 
$$ \frac{9}{\sqrt{C} m^{1- \frac{\alpha}{2}}} + \frac{16K}{C m^{2 - \alpha}} < \frac{10}{\sqrt{C} m^{1- \frac{\alpha}{2}}}$$
as needed.
\end{proof}

After this, we can prove the following Theorem:

\begin{theorem}
\label{f.p. for M^+ model}
Let $\frac{3}{2} < \alpha < 2, C >0$ be constants and $\rho (m) =  \frac{C}{m^{\alpha}}$.  Define a class of Banach spaces $\mathcal{E}_{\mathcal{M}^+} (\alpha, C,m)$ as the union of all the classes $\mathcal{E}^{\text{u-curved}}_\omega$ for which $\omega (\frac{10}{\sqrt{C} m^{1- \frac{\alpha}{2}}}) \leq \frac{1}{26}$. 

Then $\mathcal{M}^+ (m, \rho)$  has property $(F \mathcal{E}_{\mathcal{M}^+} (\alpha, C,m))$ a.a.s.  In particular:
\begin{itemize}
\item For any uniformly curved space $\B$,  $\mathcal{M}^+ (m, \rho)$ has property $(F \B)$ a.a.s. 
\item For $ \theta_0 =  \frac{\log (26)}{(1-\frac{\alpha}{2}) \log (m) + \frac{1}{2} \log (C) - \log (10)}$,   $\mathcal{M}^+ (m, \rho)$ has property $(F \mathcal{E}_{\theta_0})$ a.a.s.
\item For $2 \leq p \leq  \frac{2 (1-\frac{\alpha}{2}) \log (m) +  \log (C) - 2 \log (10)}{\log (26)}$,  $\mathcal{M}^+ (m, \rho)$ has property $(F L^p)$ a.a.s.
\end{itemize}
\end{theorem}

\begin{proof}
Let $\Gamma \in \mathcal{M}^+ (m, \rho)$ be a random group. Denote by $X$ the Cayley complex of $X$. This is a $2$-dimensional simplicial complex on which $\Gamma$ acts freely and transitively on vertices and the links of $X$ are bipartite graphs. By Theorem \ref{link of Cayley complex thm} it holds a.a.s. that the second largest eigenvalue of the random walk on the link of a vertex in $X$ is $< \frac{10}{\sqrt{C} m^{1- \frac{\alpha}{2}}}$. Thus by Corollary \ref{norm bound on rw on uc coro - bipartite} it holds a.a.s. that for every $\B \in \mathcal{E}_{\mathcal{M}^+} (\alpha, C,m)$ and every $v \in X(0)$
$$\lambda_{X_v,\bipartite}^\B < 2 \omega \left( \frac{10}{\sqrt{C} m^{1- \frac{\alpha}{2}}} \right)\leq \frac{1}{13}.$$
Thus by Theorem \ref{Zuk type thm} (noting that $\mathbb{C} \in \mathcal{E}_{\mathcal{M}^+} (\alpha, C,m)$) it holds that a.a.s. that $\Gamma$ has property $(F \mathcal{E}_{\mathcal{M}^+} (\alpha, C,m))$. 

Next, we derive the particular statements stated above. 

First, for every uniformly curved space $\B$ it holds that $\B \in \mathcal{M}^+ (m, \rho)$ given that $m$ is large enough. Thus for every uniformly curved space it holds a.a.s. that $\mathcal{M}^+ (m, \rho)$ has property $(F \B)$. 

Second,  by Corollary \ref{theta-hil is in uni-curv cor} it holds for every $0 < \theta_0 \leq 1$  that $\mathcal{E}_{\theta_0} \subseteq \mathcal{E}^{\text{u-curved}}_{\omega (t) = t^{\theta_0}}$.  Thus if 
\begin{equation}
\label{theta0 ineq}
\left( \frac{10}{\sqrt{C} m^{1- \frac{\alpha}{2}}} \right)^{\theta_0} = \frac{1}{26}
\end{equation}
it follows that $\mathcal{E}_{\theta_0} \subseteq \mathcal{E}_{\mathcal{M}^+} (\alpha, C,m)$ and therefore a.a.s.  $\mathcal{M}^+ (m, \rho)$ has property $(F \mathcal{E}_{\theta_0})$.  The equation \eqref{theta0 ineq} is equivalent to 
$$\theta_0 = \frac{\log (26)}{(1-\frac{\alpha}{2}) \log (m) + \frac{1}{2} \log (C) - \log (10)}$$
as needed.

Last, for every $2 \leq p < \infty$, every $L^p$ space is $\frac{2}{p}$-strictly Hilbertian. Thus for $\theta_0 =\frac{\log (26)}{(1-\frac{\alpha}{2}) \log (m) + \frac{1}{2} \log (C) - \log (10)}$ as above, if $\frac{2}{p} \geq \theta_0$ it follows that 
$$\frac{2}{p} \leq \frac{\log (26)}{(1-\frac{\alpha}{2}) \log (m) + \frac{1}{2} \log (C) - \log (10)}$$
it follows that a.a.s. $\mathcal{M}^+ (m, \rho)$ has property $(F L^p)$. The condition $\frac{2}{p} \geq \theta_0$ is equivalent to
$$p \leq  \frac{2 (1-\frac{\alpha}{2}) \log (m) +  \log (C) - 2 \log (10)}{\log (26)}$$
as needed.

\end{proof}

Combining the above Theorem and Theorem \ref{reduction thm} allows us to deduced fixed point properties for random groups in the Gromov density model:
\begin{theorem}
\label{f.p. density model}
Let $\frac{1}{3} <d < \frac{1}{2}$, $k \geq 2$ be constants. Define a class of Banach spaces $\mathcal{E}_{\mathcal{D}} (k,l,d)$ as the union of all the classes $\mathcal{E}^{\text{u-curved}}_\omega$ for which $\omega \left(\frac{20 \sqrt{2}}{(2k-1)^{\frac{l}{6} (3d - 1)}} \right) \leq \frac{1}{26}$. 

Then for $l$ divisible by $3$, $\mathcal{D} (k,l,d)$ has property $(F \mathcal{E}_{\mathcal{D}} (k,l,d))$ a.a.s.  In particular:
\begin{itemize}
\item For any uniformly curved space $\B$,  $\mathcal{D} (k,l,d)$ has property $(F \B)$ a.a.s.   given that $l$ is divisible by $3$.
\item For $ \theta_0 =  \frac{\log (26)}{\frac{l}{6} (3d-1) \log (2k-1) - \log (20 \sqrt{2})}$,   $\mathcal{D} (k,l,d)$ has property $(F \mathcal{E}_{\theta_0})$ a.a.s.  given that $l$ is divisible by $3$.
\item For $2 \leq p \leq   \frac{l (d-\frac{1}{3}) \log (2k-1) - 2 \log (20 \sqrt{2})}{\log (26)}$,  $\mathcal{D} (k,l,d)$ has property $(F L^p)$ a.a.s.  given that $l$ is divisible by $3$.
\end{itemize}
\end{theorem}

\begin{proof}
Let $m = m(k,l) =  \frac{1}{2} (2k-1)^{\frac{l}{3}}$ and $\rho (m) = \frac{1}{4} \frac{1}{m^{3(1-d)}}$. We can write 
$\rho (m) = C  \frac{1}{m^{\alpha}}$ for $C =\frac{1}{4}, \alpha = 3(1-d)$.  Then with this $m$ and $\rho$,  the class $\mathcal{E}_{\mathcal{M}^+} (\alpha, C,m)$ defined in Theorem \ref{f.p. for M^+ model} contains the class $\mathcal{E}_{\mathcal{D}} (k,l,d)$ defined above.   By Theorem \ref{reduction thm} for $l$ divisible by $3$, $\mathcal{D} (k,l,d)$ has property $(F \mathcal{E}_{\mathcal{D}} (k,l,d))$ a.a.s.  


The particular statements are proven as in Theorem \ref{f.p. for M^+ model} and are left for the reader.
\end{proof}

\begin{remark}
By the proof of Theorem \ref{reduction thm} a similar Theorem can also be proved in the Gromov binomial model and this is left for the reader.
\end{remark}

\bibliographystyle{alpha}
\bibliography{bib1}
\Addresses
\end{document}